\documentclass[11pt,reqno]{amsproc}
\usepackage[margin=1in]{geometry}
\usepackage{amsmath, amsthm, amssymb}
\usepackage[colorlinks=true, pdfstartview=FitV, linkcolor=blue,citecolor=blue, urlcolor=blue]{hyperref}
\usepackage[abbrev,lite,nobysame]{amsrefs}
\usepackage{times}
\usepackage[usenames,dvipsnames]{color}
\usepackage{bbm}
\usepackage{mathtools}
\usepackage{subcaption}

\mathtoolsset{showonlyrefs=true}
\usepackage{dsfont}

\def\eps{\varepsilon}
\def\e{{\rm e}}
\def\E{{\mathbb E}}

\def\dd{{\rm d}}

\def\R {\mathbb{R}}
\def\N {\mathbb{N}}

\def\T {{\mathbb T}}
\def\de{{\partial}}
\newcommand\uu {\boldsymbol{u}}

\newcommand{\be}{\begin{equation}}
\newcommand{\ee}{\end{equation}}
\newcommand{\bea}{\begin{eqnarray}}
\newcommand{\eea}{\end{eqnarray}}

\newcommand{\rmd}{{\rm d}}
\newcommand{\spt}{{\rm spt}}
\newcommand{\bx}{{ \boldsymbol{x} }}
\newcommand{\bxi}{{ \boldsymbol{X}} }

\newcommand{\bq}{\begin{equation}}
\newcommand{\eq}{\end{equation}}
\newcommand{\bqa}{\begin{eqnarray*}}
\newcommand{\eqa}{\end{eqnarray*}}

\newtheorem{theorem}{Theorem}

\newtheorem{lemma}{Lemma}
\theoremstyle{definition}
\newtheorem{definition}{Definition}
\newtheorem{remark}{Remark}

\numberwithin{equation}{section}

\title[A stochastic approach to enhanced diffusion]{A stochastic approach to enhanced diffusion}
\author[M. Coti Zelati]{Michele Coti Zelati}
\address{Department of Mathematics, Imperial College London, London, SW7 2AZ, UK}
\email{m.coti-zelati@imperial.ac.uk}

\author[T. D. Drivas]{Theodore D. Drivas}
\address{Department of Mathematics, Princeton University, Princeton, NJ 08544}
\email{tdrivas@math.princeton.edu}



\begin{document}

\begin{abstract}
We provide examples of initial data which saturate the enhanced diffusion rates proved for general shear flows which are H\"{o}lder regular or Lipschitz continuous with critical points, and for regular circular flows, establishing the sharpness of those results.  Our proof makes use of a probabilistic interpretation of the dissipation of solutions to the advection diffusion equation. 
\end{abstract}


\maketitle


\section{Introduction}

Let $\Omega$ be a two-dimensional domain without boundary, i.e. $\Omega=\R^2$ or $\Omega=\T\times D$, where $D=\T$ or $D= \R$.
Given an incompressible autonomous velocity field $\uu:\Omega\to\R$, it is a classical problem to understand 
the longtime dynamics of a passive tracer $\rho:\mathbb{R}^+\times \Omega\to\R$ satisfying the linear drift-diffusion equation
\begin{align}\label{eq:DDE}
\partial_t \rho + \uu\cdot\nabla \rho = \kappa\Delta \rho,
\end{align}
for a given initial condition
\begin{align}
\rho(0,\bx)=\rho_0(\bx),\qquad \bx\in \Omega.
\end{align}
Above, $\kappa\in(0,1)$ is a small constant molecular diffusivity. 
In this paper, we assume without loss of generality that $\rho_0$ is mean-free, which immediately implies that
\begin{align}\label{meanfree}
\int_{\Omega}\rho(t,\bx)\dd \bx=0, \qquad \forall t\geq0.
\end{align}
It is by now well-established that the longtime behavior of $\rho$ can be described by an effective diffusion equation \cites{RhinesYoung83,vukadinovic2015averaging,kumaresan2018advection,AM91,FW12,FW94,BL94}, at least
under non-degeneracy and growth assumptions on $\uu$. This equation accurately describes the solution to \eqref{eq:DDE} at time-scales that
are longer than the purely diffusive one, proportional to $1/\kappa$.  Letting $\uu=\nabla^\perp \psi$, where $\psi$ is
the stream-function generating $\uu$, the effective diffusion takes place \emph{across} streamline, once $\rho$ has been mixed enough so that 
it is almost constant on the level sets of $\psi$.

On shorter time-scales, it is fairly accurate to think of advection as the dominating process. In this case, the main effect to capture is that 
of mixing \emph{along} streamlines. This concept was formalized mathematically in \cite{CKRZ08} in terms of \emph{enhanced diffusion}
time-scales, faster than $1/\kappa$. See Figures \ref{fig:shear} and \ref{fig:coffee} for a snapshots of a numerical simulation which illustrate this multi-scale phenomenon in the two settings which we study in depth in this paper.
A quantitative definition of enhanced diffusion can be phrased as follows.

\begin{definition}\label{def:enhanced}
Let $\kappa_0\in(0,1)$ and $r:(0,\kappa_0]\to (0,1)$ be a continuous increasing function such that
\begin{align}
\lim_{\kappa\to 0}\frac{\kappa}{r(\kappa)}=0.
\end{align}
The velocity field $\uu$ is diffusion enhancing on a subspace $H\subset L^2$ at rate $r(\kappa)$ if
there exists $C\geq 1$ only depending on $\uu$ such that if $\kappa\in(0, \kappa_0]$ then for every $\rho_0\in H$ we have the enhanced diffusion estimate
\begin{align}\label{eq:endecay}
\|\rho(t)\|^2_{L^2}\leq C \e^{- r(\kappa)t}\|\rho_0\|^2_{L^2}, \qquad \forall t\geq 0.
\end{align}
\end{definition}

\begin{remark}\label{remstream}
The subspace $H$ is just to say that one must avoid data concentrated on level sets of streamlines. Indeed, if the level sets of $\psi$ are closed curves foliating a region of positive area, then one must require that the average of the data along the streamlines
be zero, since convection can have no effect for the values on such sets (the velocity is necessarily tangent to the streamlines).  This is most easily seen for shear flows where the streamlines are
parametrized by values of $y$ and thus one assumes
that the $x$-average for each given $y$ of the data is zero.
\end{remark}

Quantitative enhanced dissipation rates for passive scalars have been studied for shear flows and radially symmetric flows in  
\cites{BCZGH15,BCZ15,BW13, CZDE18,CZD19, WEI18, FI19}.
In view of the incompressibility and smoothness of $\uu$, solutions to \eqref{eq:DDE} satisfy the energy equality
\begin{align}\label{eq:enreq}
\frac12 \|\rho(t)\|^2_{L^2}+\kappa \int_0^t \|\nabla \rho(s)\|_{L^2}^2\rmd s =\frac12\|\rho_0\|^2_{L^2}, \qquad \forall t\geq 0.
\end{align}
In this note, we draw an elementary connection between upper bounds on the dissipation of the form
\begin{align}\label{eq:updisIntro}
\kappa \int_0^t \|\nabla \rho(s)\|_{L^2}^2 \rmd s\leq f(\xi(\kappa) t)\|\rho_0\|^2_{L^2},
\end{align}
 for some continuous increasing function $f:[0,\infty)\to[0,\infty)$ with $f(0)=0$  and some 
increasing function $\xi:(0,\kappa_0]\to (0,1)$  to constraints on the rate of enhanced diffusion $r(\kappa)$ as in Definition \ref{def:enhanced}.
 This connection is made precise in Lemma \ref{lem:crucial}. 

The utility of such a statement is complemented by a simple method to estimate the size of the dissipation by means of a stochastic interpretation.
First, if one introduces the backward trajectories
\be\label{eq:backtraj}
\rmd \bxi_{t,s}(\bx) = \uu( \bxi_{t,s}(\bx) )\rmd s+ \sqrt{2\kappa} \ \hat{\rmd} \begin{pmatrix}B_s\\ W_s\end{pmatrix}, \qquad  \bxi_{t,t}(\bx) =\bx,
\ee
then the well-known Feynman-Kac formula \cite{kunita1997stochastic} expresses the solution of the advection-diffusion equation as an average of the data sampled on the above trajectories
\be\label{feykac}
\rho(t,\bx) = \E  \left[ \rho_0(\bxi_{t,0}(\bx)) \right].
\ee
In the above, $B_s$ and $W_s$ are $\mathbb{R}$ valued independent Brownian motions adapted to the backwards filtration, i.e. satisfying $B_t=W_t =0$ and the expectation is over both processes.  The $\hat{\rmd} $ in \eqref{eq:backtraj} implies a backward It\^{o} SDE.
For a detailed discussion of backward  It\^{o} equations see \cite{kunita1997stochastic}.
The idea in our work is to use the so-called Lagrangian fluctuation-dissipation relation introduced in \cite{drivas2017lagrangian,drivas2017lagrangian2} which directly provides a way to express the dissipation of the scalar field in terms of these same stochastic trajectories
\be\label{eq:FDR}
\kappa \int_0^t  \|\nabla \rho(s)\|_{L^2}^2 \rmd s =\frac{1}{2} \int_{\Omega} {\rm Var}\left(\rho_0(\bxi_{t,0}(\bx)\right)\rmd \bx.
\ee
For completeness, we include a derivation or the above relation in our setting in \S \ref{secFDR}.

\subsection{Shear flows}
In the spatial setting $\Omega=\T\times D$, where $D=\T$ or $D= \R$,
consider the special case when the velocity vector field is a \emph{shear flow}, namely $\uu(x,y)=(u(y),0)$. Specifically, 
given $\delta\in\{0,1\}$ and defining $\Delta_\delta=\delta\de_{xx}+\de_{yy}$, we consider the equation
\begin{align}\label{eq:shear}
\begin{cases}
\partial_t \rho + u\de_x\rho = \kappa\Delta_\delta \rho, \\
\rho(0)=\rho_0.
\end{cases}
\end{align}
Notice that when $\delta=0$ the equation is only partially diffusing, however all the results stated here are not affected by this change.
Since the $x$-average of the solution satisfies the one-dimensional heat equation in $y$, we further restrict the initial data
to be such that
\begin{align}
\int_{\T}\rho_0(x,y)\dd x=0,\qquad \text{for a.e. }  y\in D.
\end{align}
In this way, the scalar has zero average on streamlines (constant $y$-lines) and this condition is propagated
\begin{align}
\int_{\T}\rho(t,x,y)\dd x=0, \qquad \forall t\geq0, \  \text{for a.e. }  y\in D.
\end{align}
In this case, the subspace of Definition \ref{def:enhanced} is
\begin{align}
H=\left\{\rho\in L^2:\ \int_{\T}\rho(x,y)\dd x=0,\ \text{for a.e. }  y\in D\right\}.
\end{align}
See also Remark \ref{remstream}.
From the point of view of enhanced diffusion, the following two theorems, proven in \cite{BCZ15} 
and \cite{WEI18} respectively, are of interest to us.

\medskip

\noindent\textbf{Regular shear flows with critical points.}
Assume that $u \in C^{n+1}(\Omega)$ has a \emph{finite} number
of critical points (generically denoted by $y_{crit}$), where  $n\in \N$ 
denotes the maximal order of vanishing of $u'$ at the critical points, namely, the minimal integer
such that
\begin{equation}\label{eq:last}
u^{(n)}(y_{crit})\neq 0,
\end{equation}
for every $y_{crit}$. Essentially, $|u(y)-u(y_{crit})|\sim |y-y_{crit}|^n$ near the critical points of maximal order. 
In the case of the unbounded domain $\T\times \R$, we assume that $u'$ and $u''$ are bounded functions,
although more general shears can be treated by using weighted spaces \cite{CZ19}.
Then we have the following result. 

\begin{theorem}[\cite{BCZ15}*{Theorem 1.1}]\label{thm:firstthm}
Let $\delta\in\{0,1\}$.
There exist $\eps\in(0,1)$ only depending on $u$ and $\delta$ such that the enhanced diffusion rate on $H$ is
\begin{align}\label{eq:decayrate}
r(\kappa) =\eps \frac{\kappa^{\frac{n}{n+2}}}{1 +\log \kappa^{-1}}.
\end{align}
If the spatial domain is $\T\times\R$ and $u'$ is bounded below, then we have $r(\kappa)=\eps \kappa^{1/3}$.
\end{theorem}

\begin{figure}[h!]
  \centering
  \begin{subfigure}[b]{0.24\linewidth}
    \includegraphics[width=\linewidth]{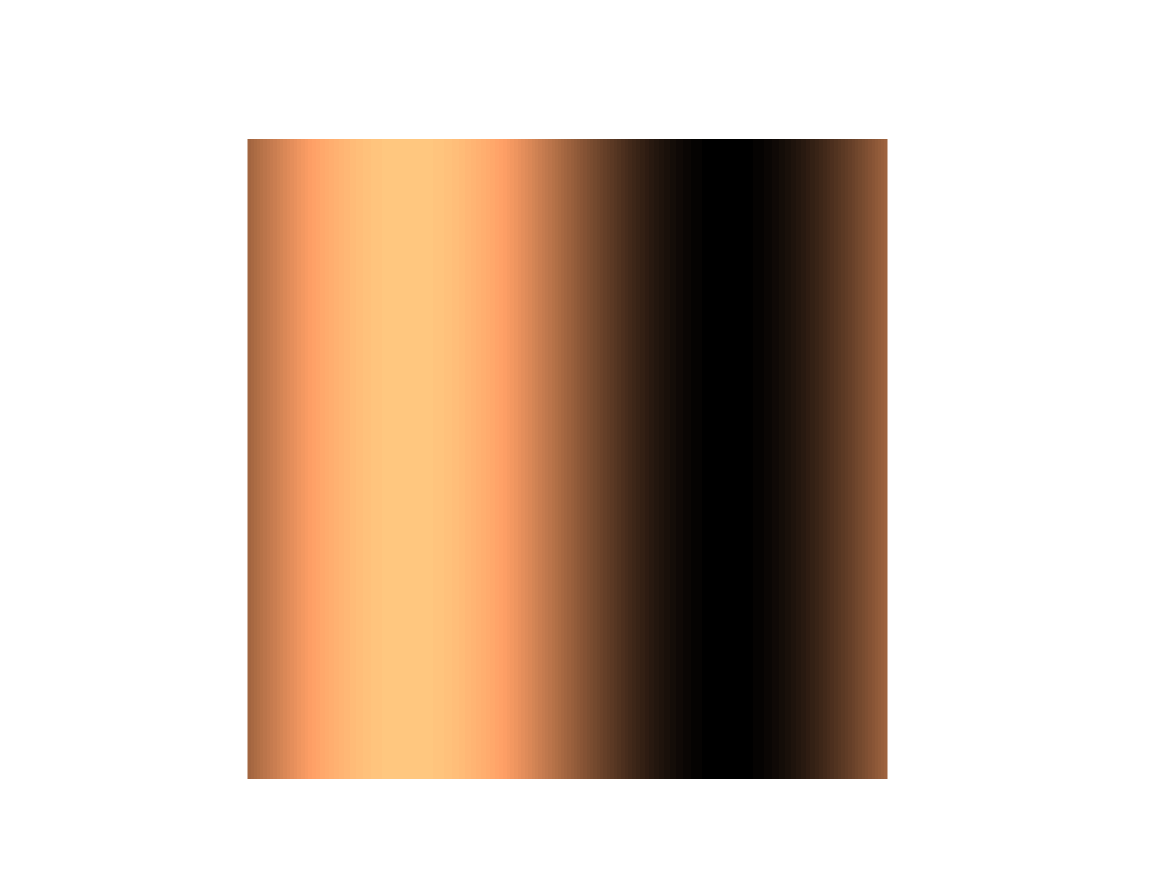}
  \end{subfigure}
  \begin{subfigure}[b]{0.24\linewidth}
    \includegraphics[width=\linewidth]{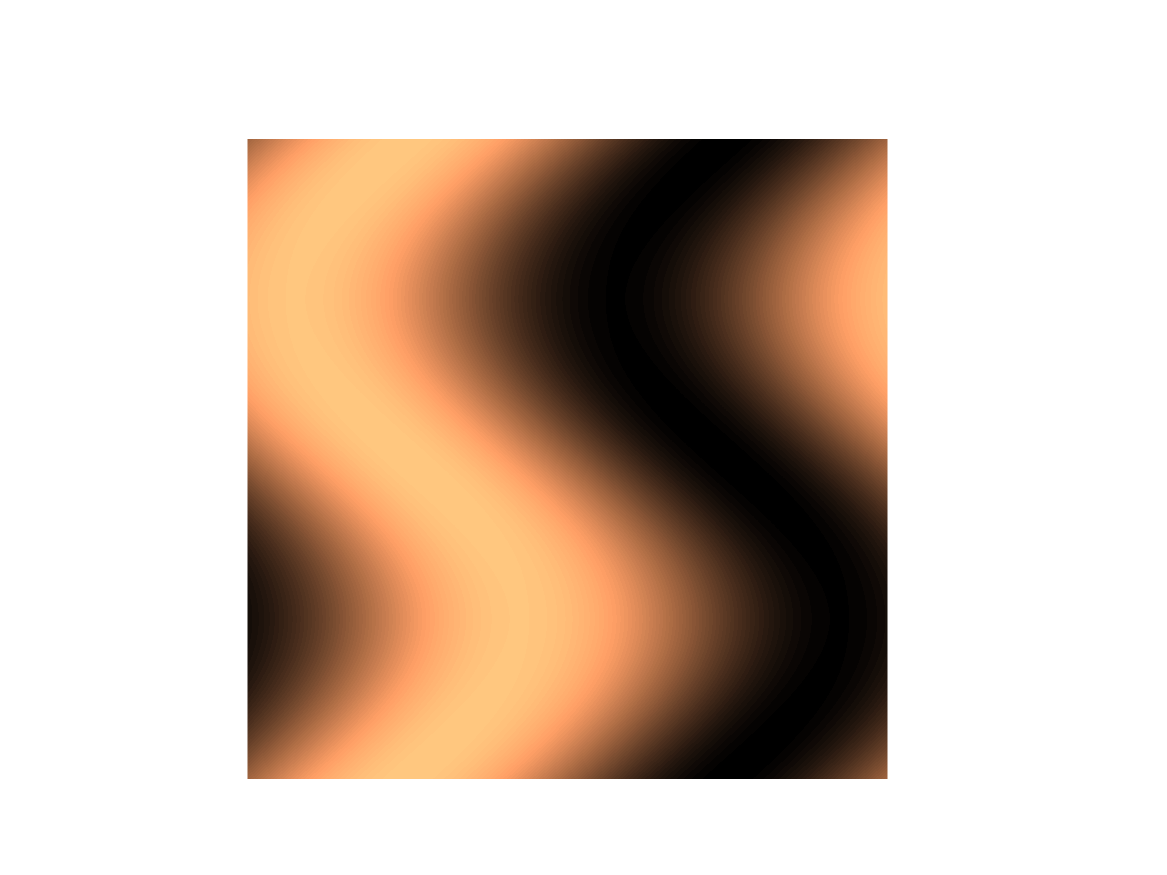}
  \end{subfigure}
  \begin{subfigure}[b]{0.24\linewidth}
    \includegraphics[width=\linewidth]{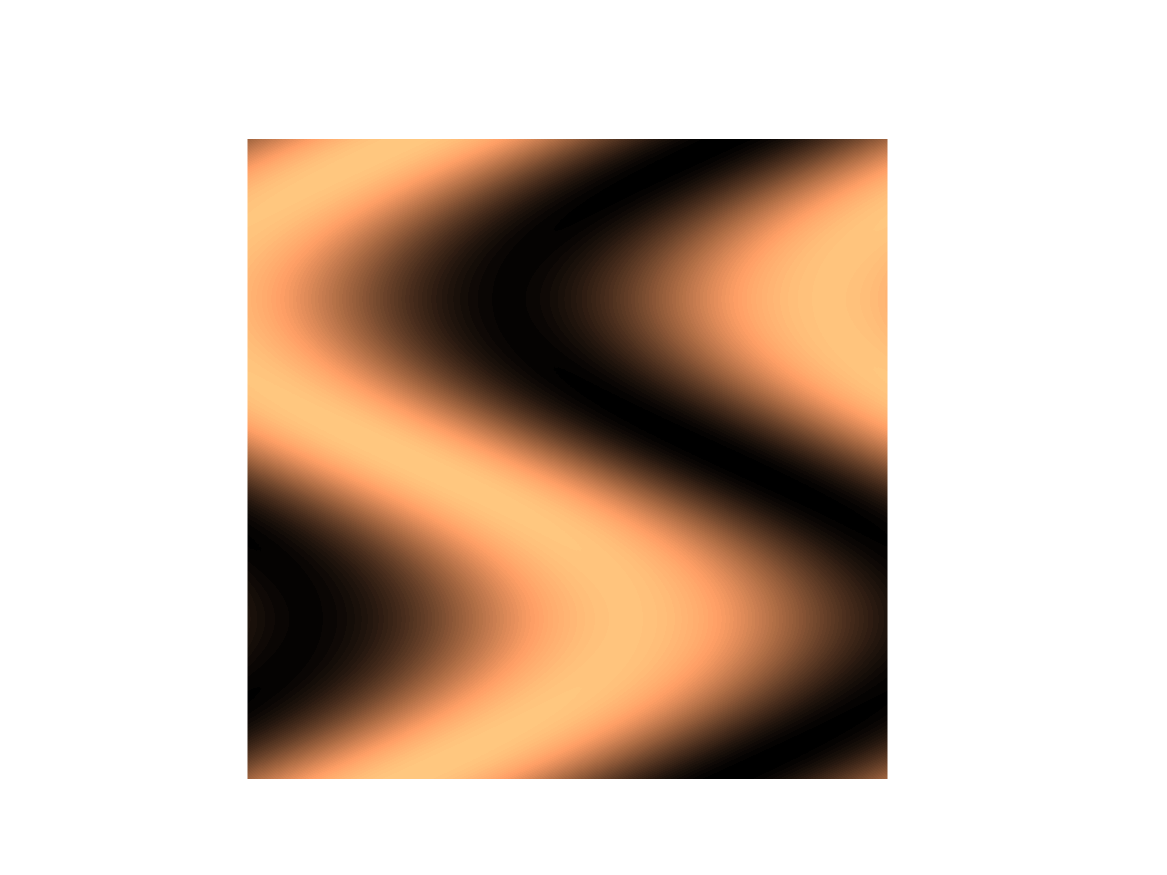}
  \end{subfigure}
  \begin{subfigure}[b]{0.24\linewidth}
    \includegraphics[width=\linewidth]{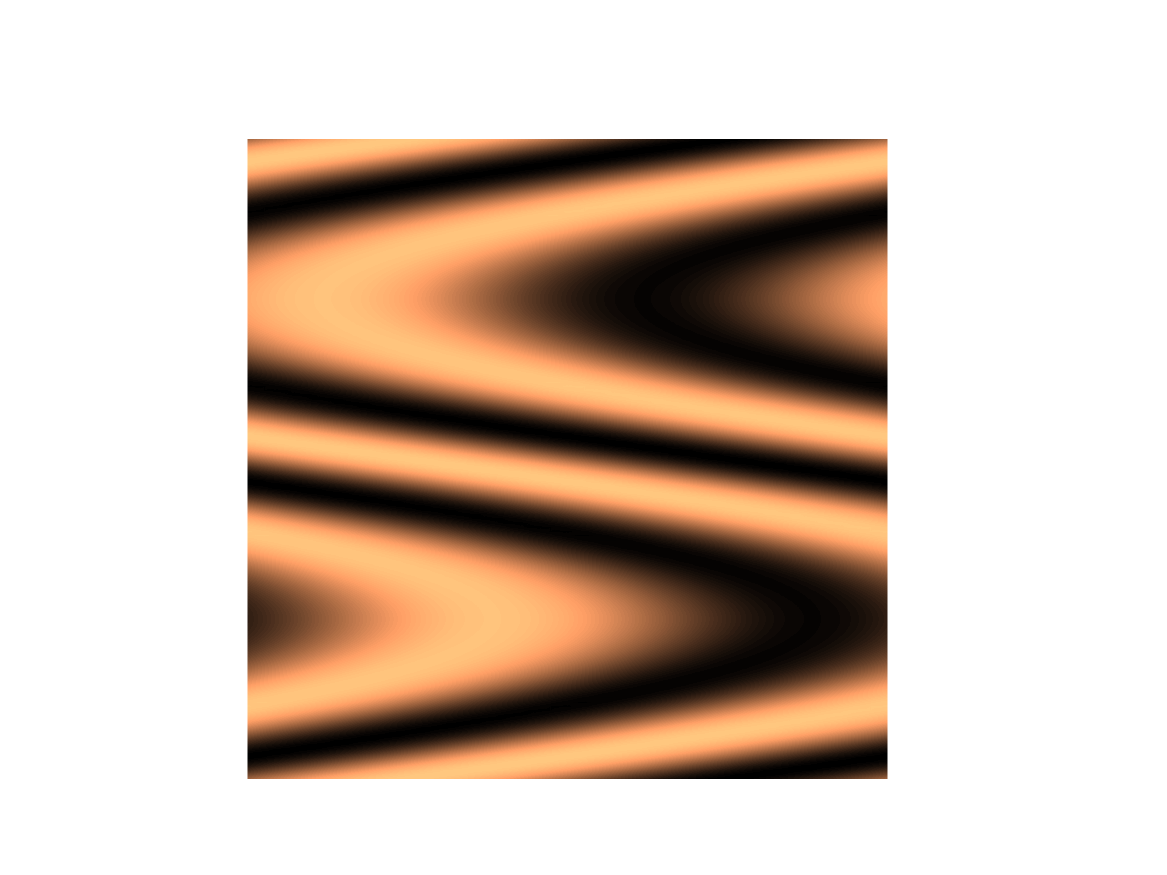}
  \end{subfigure}
    \begin{subfigure}[b]{0.24\linewidth}
    \includegraphics[width=\linewidth]{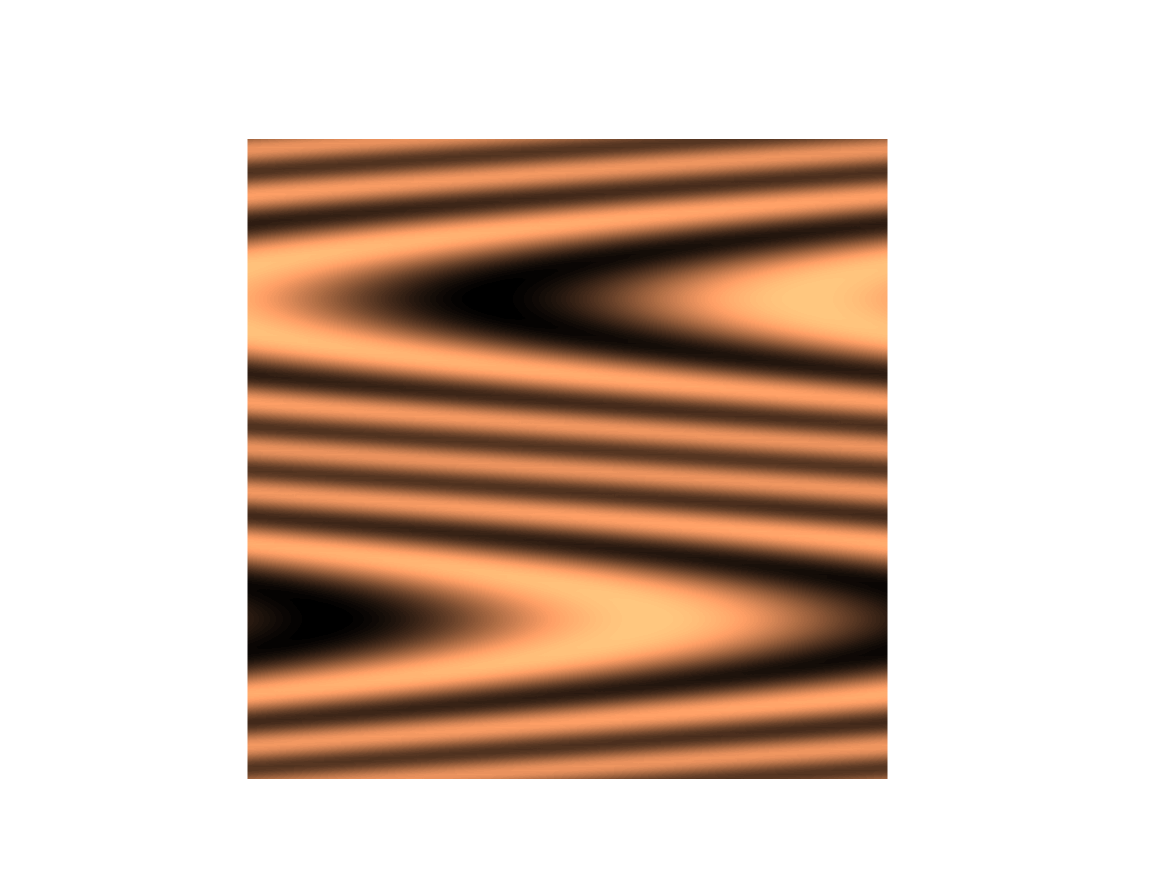}
  \end{subfigure}
  \begin{subfigure}[b]{0.24\linewidth}
    \includegraphics[width=\linewidth]{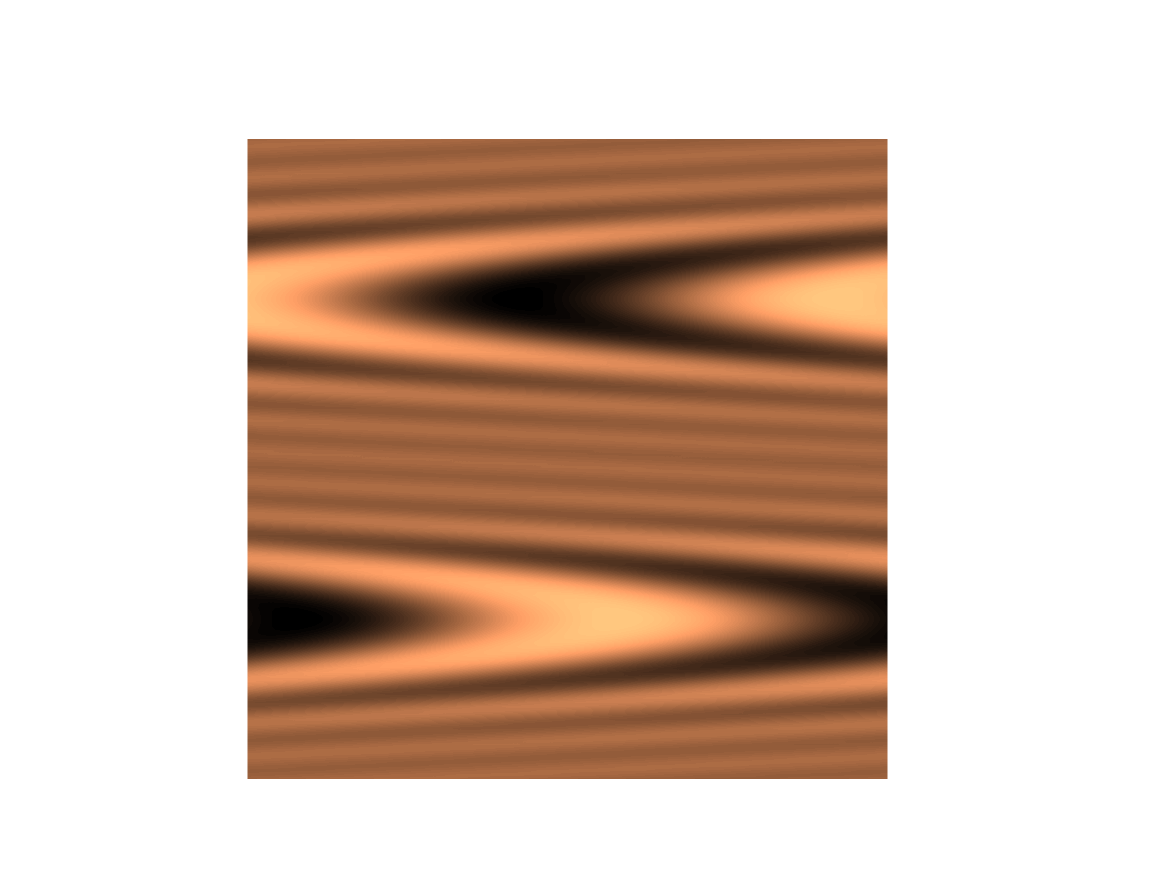}
  \end{subfigure}
  \begin{subfigure}[b]{0.24\linewidth}
    \includegraphics[width=\linewidth]{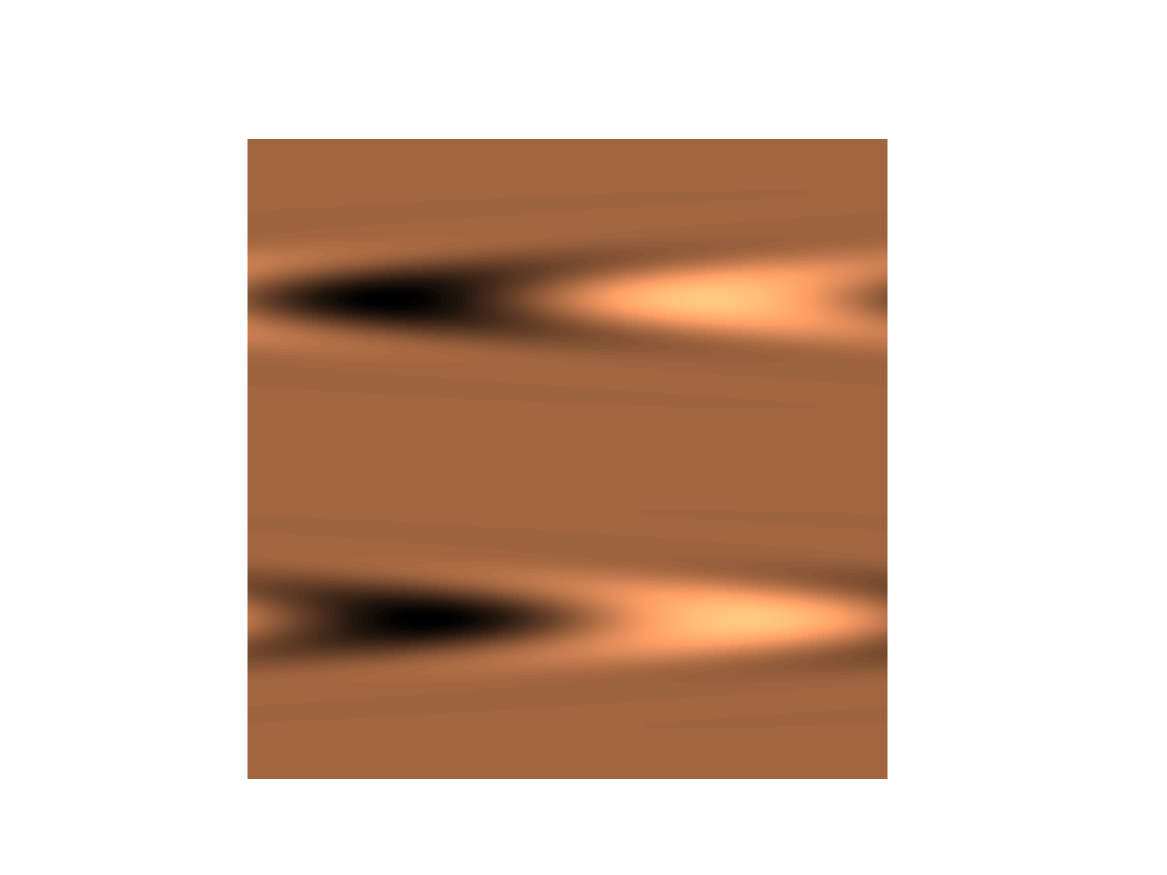}
  \end{subfigure}
  \begin{subfigure}[b]{0.24\linewidth}
    \includegraphics[width=\linewidth]{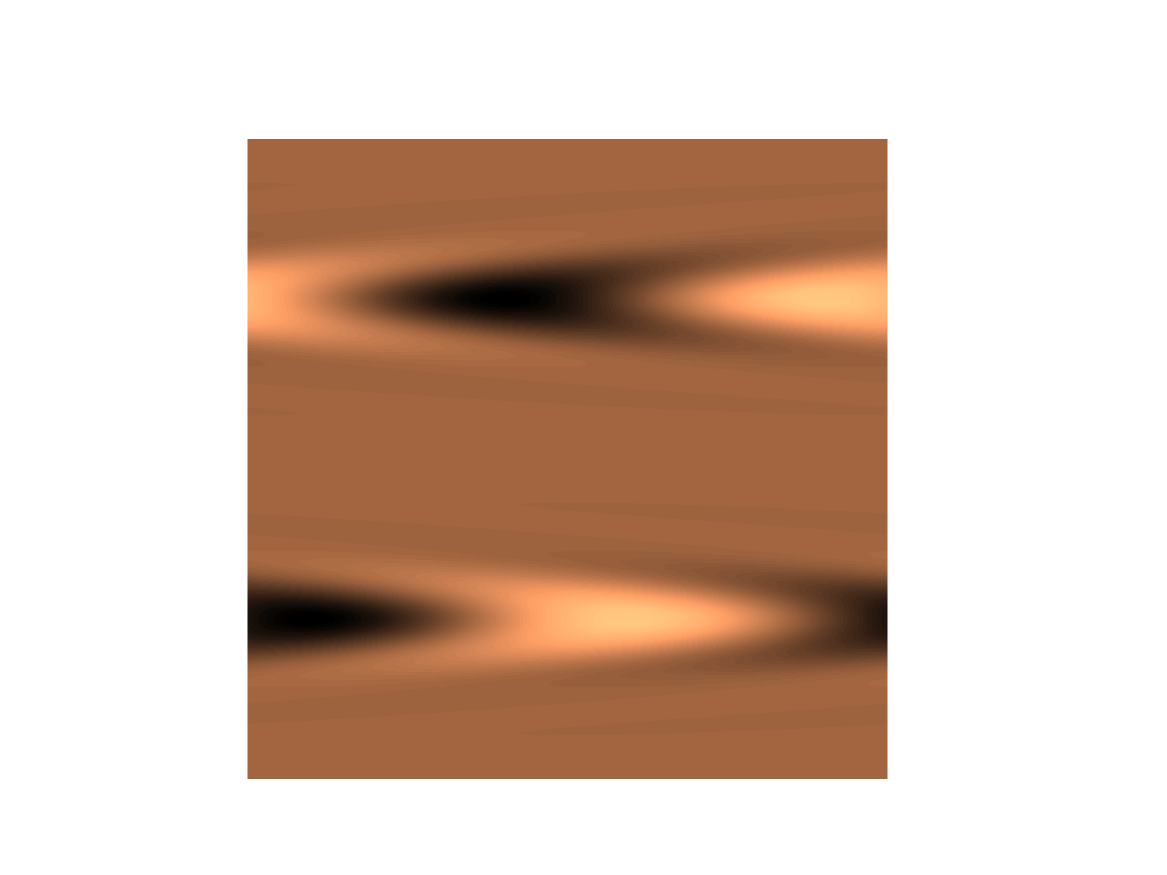}
  \end{subfigure}
  \caption{The evolution of a strip of diffusive ``paint"  sheared in a periodic box by the flow $u(y)=\sin y$. Advection dominates until the paint becomes independent of $x$ in regions away from the two critical points. After this time, the paint diffuses across the (horizontal) streamlines.}
  \label{fig:shear}
\end{figure}

We remark that Theorem \ref{thm:firstthm} entails a semigroup estimate for the linear semigroup generated by \eqref{eq:shear}
and restricted to the invariant subspace $H$. In particular, the initial data $\rho_0$ may be taken to depend on $\kappa$. 
In addition, the logarithmic correction in \eqref{eq:decayrate} is believed to be an artifact of the proof, and it is not present in the 
uniformly monotone case (which corresponds to $n=1$).

\medskip

\noindent\textbf{H\"older shear flows.}
On the periodic domain $\Omega=\T^2$, let $\alpha\in(0,1)$ and consider the Weierstrass function
\begin{align}\label{eq:Wei}
u(y)=\sum_{\ell\geq1} \frac{1}{3^{\alpha\ell}} \sin (3^\ell y).
\end{align}
In this case, $u\in C^{\alpha}$ at every point, and no better. In this case, diffusion can be enhanced at an arbitrarily small polynomial
time-scale, depending on the parameter $\alpha$.
\begin{theorem}[\cite{WEI18}*{Theorem 5.1}]\label{thm:secondthm}
Let $\delta\in\{0,1\}$ and  $\rho_0\in H$.
There exists $\eps\in(0,1)$ only depending on $u$ and $\delta$ such that the enhanced diffusion rate of $u$ in \eqref{eq:Wei}  on $H$ is
\begin{align}\label{eq:decayrate2}
r(\kappa) =\eps \kappa^\frac{\alpha}{\alpha+2}.
\end{align}
\end{theorem}
The proofs of \cite{BCZ15}*{Theorem 1.1} and \cite{WEI18}*{Theorem 5.1} are carried out with completely different methods. The first one 
uses an adaptation of the so-called hypocoercivity technique \cite{Villani09}, while the second one studies the resolvent of  the linear operator 
appearing in \eqref{eq:shear}.

\subsection{Circular flows}
Another interesting case, analyzed in \cite{CZD19} from the point of view of enhanced dissipation, is 
that of circular flows on $\Omega=\R^2$. Assume that the velocity field
$\uu:\R^2\to \R^2$ is given by
\begin{equation}\label{def:u}
	\uu(x,y)=\left(x^2+y^2\right)^{q/2}\begin{pmatrix}
	-y\\ x
	\end{pmatrix},
\end{equation}	
with $q\geq 1$ an arbitrary fixed exponent.
In this way, the background velocity field generates a counter-clockwise rotating motion, with a shearing effect across streamlines.
By passing to polar coordinates $(r,\theta)\in (0,\infty)\times \T$ in  \eqref{eq:DDE}, we deduce that
\begin{align}\label{eq:cauchycirc}
\begin{cases}
\de_t \rho+ r^{q} \de_\theta \rho=\kappa\Delta \rho,\\ 
\rho(0,r,\theta)=\rho_0.
\end{cases}
\end{align}

\begin{figure}[h!]
  \centering
  \begin{subfigure}[b]{0.24\linewidth}
    \includegraphics[width=\linewidth]{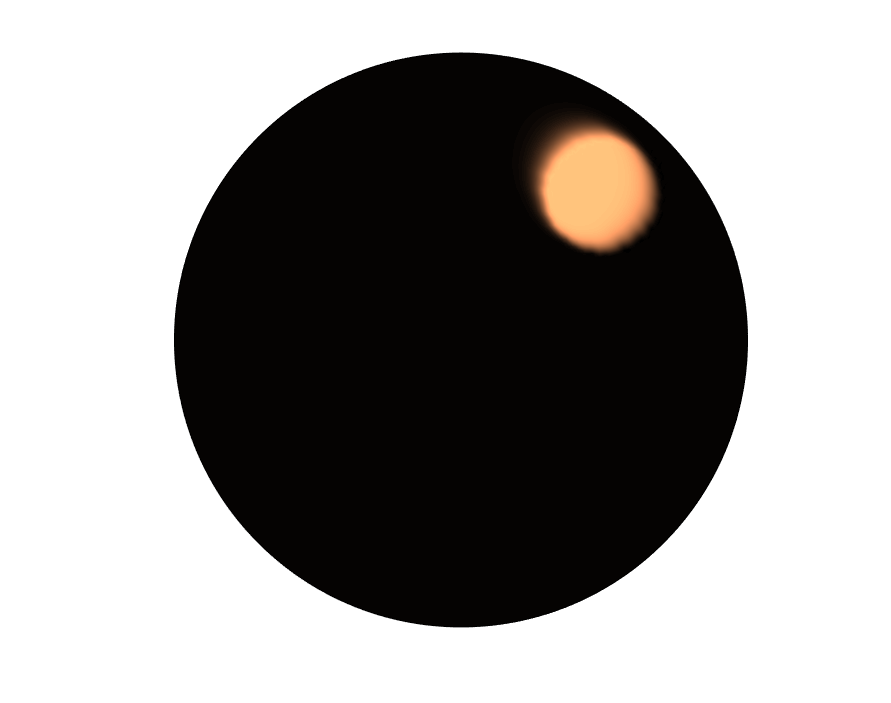}
  \end{subfigure}
  \begin{subfigure}[b]{0.24\linewidth}
    \includegraphics[width=\linewidth]{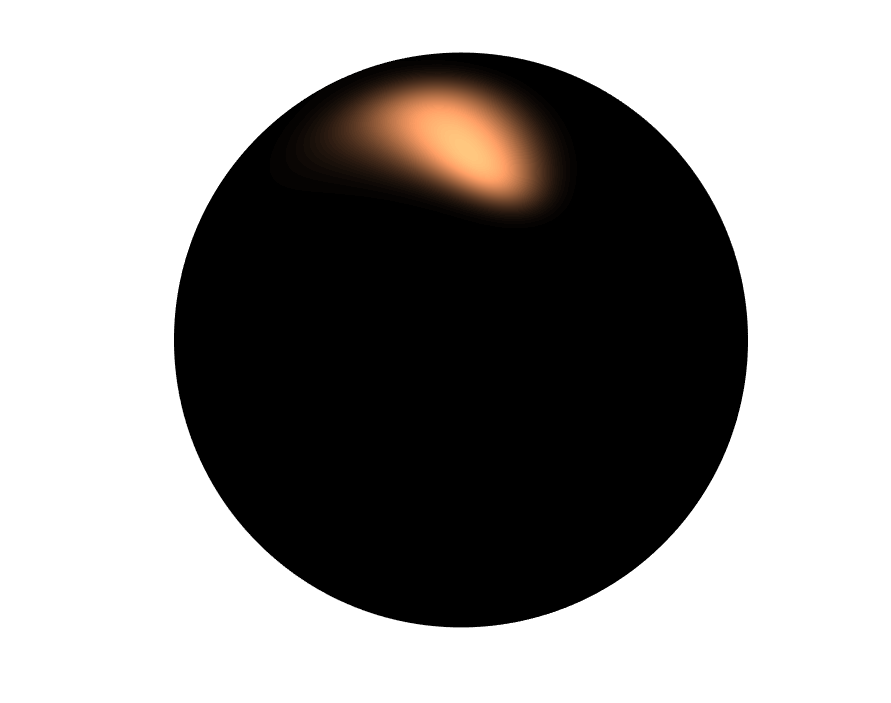}
  \end{subfigure}
  \begin{subfigure}[b]{0.24\linewidth}
    \includegraphics[width=\linewidth]{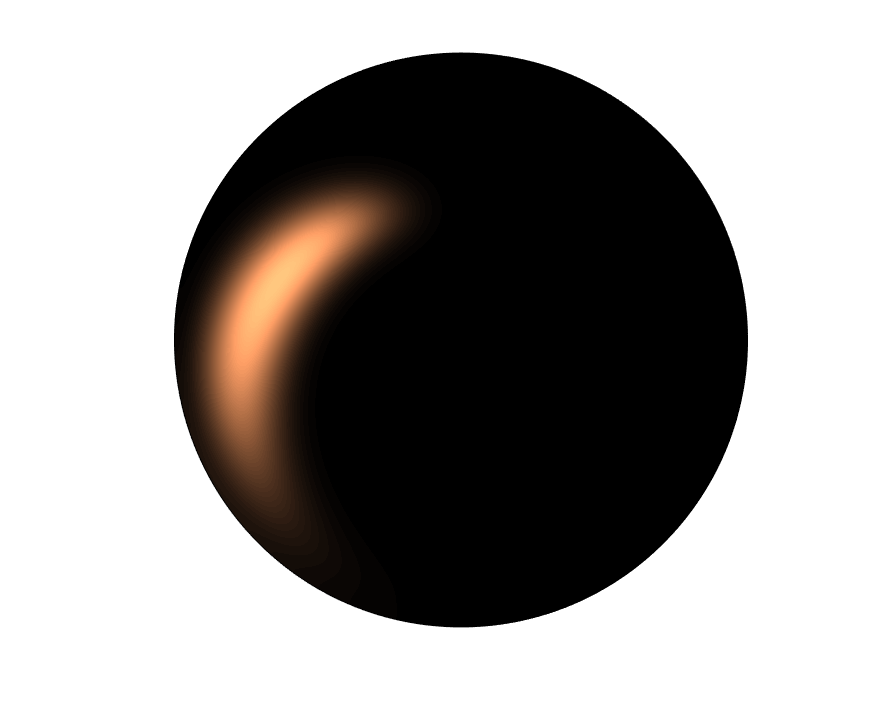}
  \end{subfigure}
  \begin{subfigure}[b]{0.24\linewidth}
    \includegraphics[width=\linewidth]{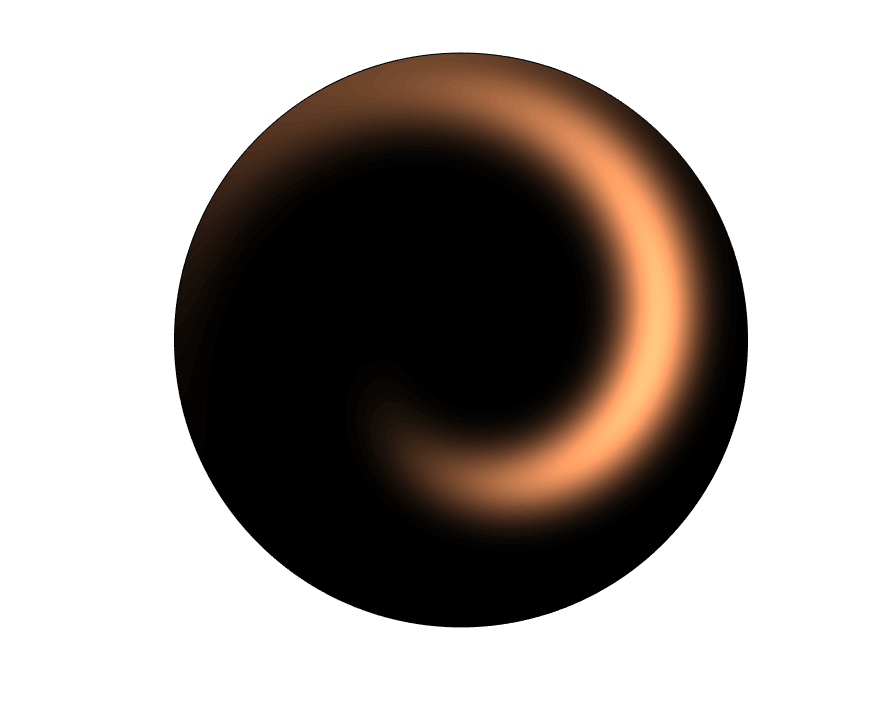}
  \end{subfigure}
    \begin{subfigure}[b]{0.24\linewidth}
    \includegraphics[width=\linewidth]{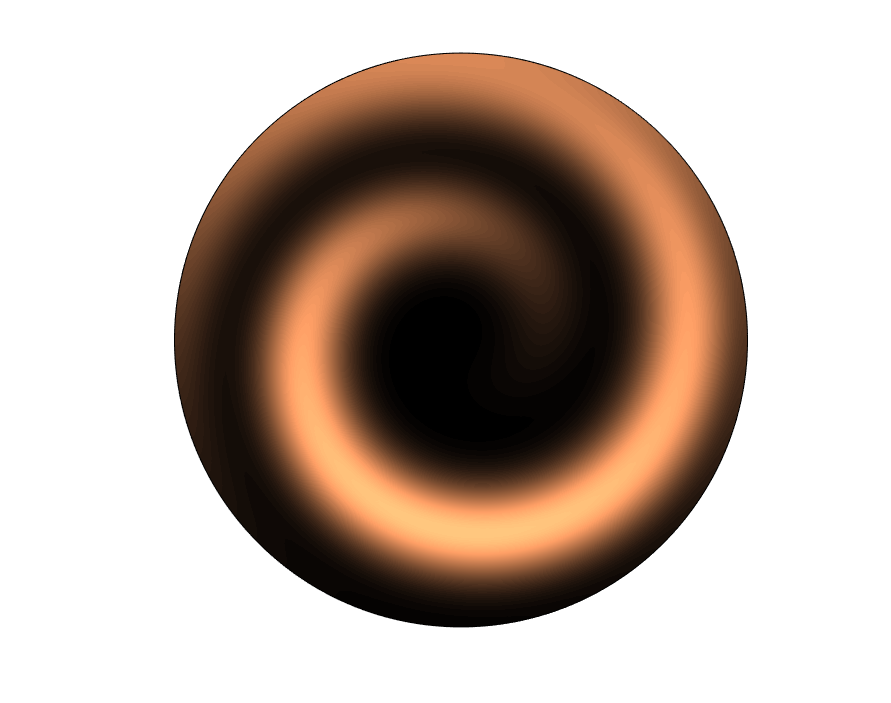}
  \end{subfigure}
  \begin{subfigure}[b]{0.24\linewidth}
    \includegraphics[width=\linewidth]{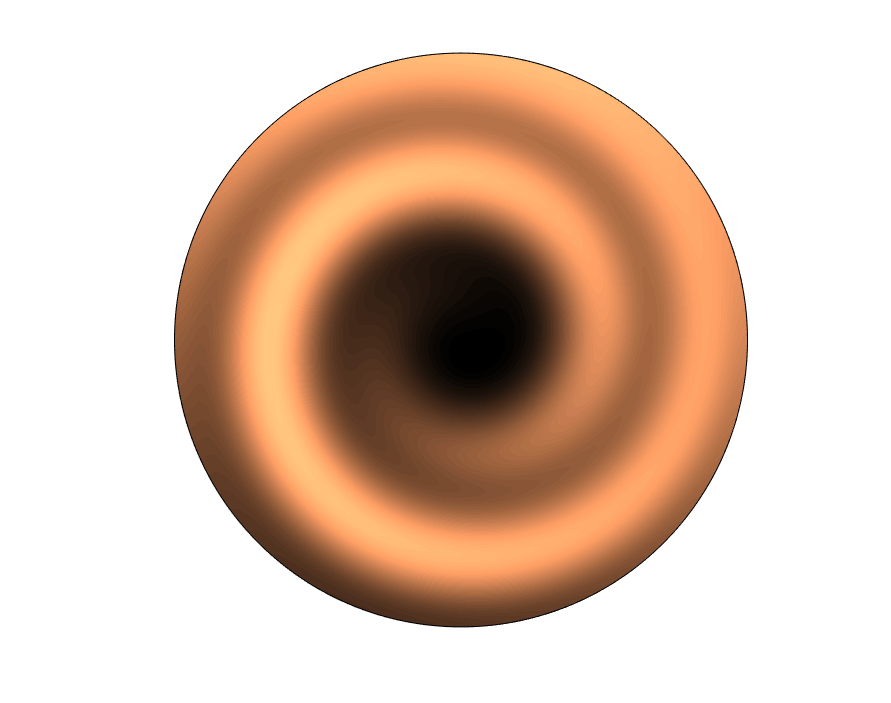}
  \end{subfigure}
  \begin{subfigure}[b]{0.24\linewidth}
    \includegraphics[width=\linewidth]{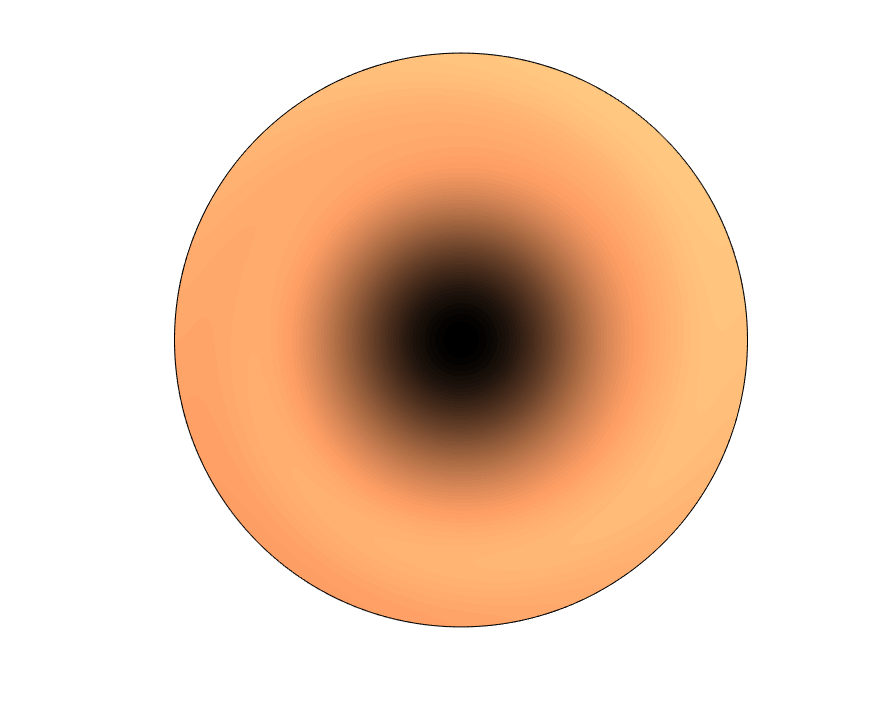}
  \end{subfigure}
  \begin{subfigure}[b]{0.24\linewidth}
    \includegraphics[width=\linewidth]{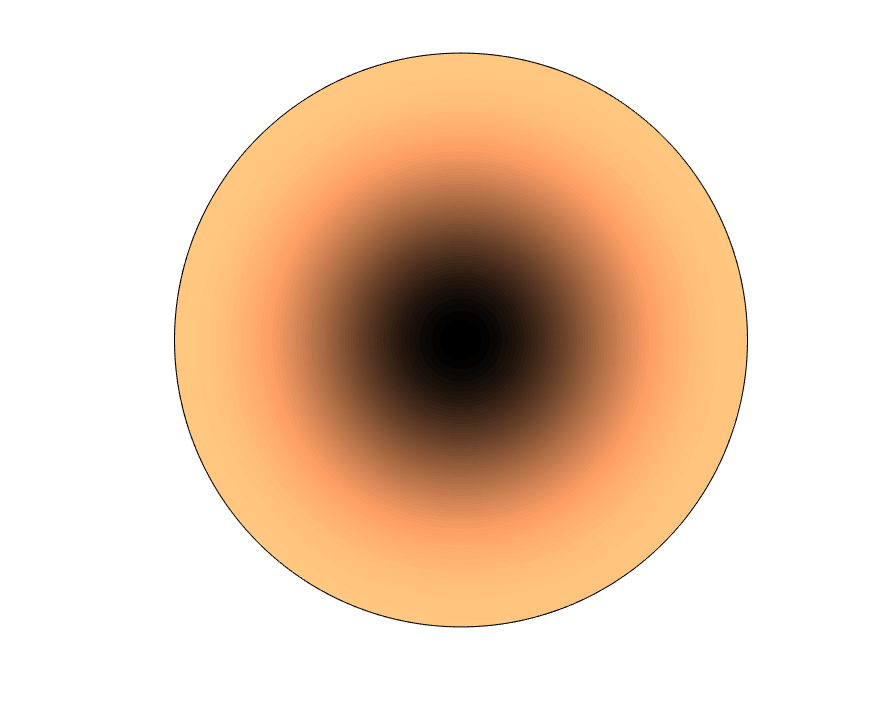}
  \end{subfigure}
  \caption{The evolution of a drop of slightly diffusive ``cream"  radial stirred into a ``cup of coffee" with impermeable walls. Initially, pure advection is the dominant
effect.  As time progresses, the solution becomes radially symmetric. After this time, the cream simply diffuses across the (circular) streamlines. 
Here $q=2$ in \eqref{def:u}. Figure taken from \cite{CZD19}.}
  \label{fig:coffee}
\end{figure}

When referring to circular flows, $\Delta$ denotes the Laplace operator in polar coordinates, namely
\begin{align}\label{laplacepolar}
\Delta=\de_{rr}+\frac{1}{r} \de_r +\frac{1}{r^2}\de_{\theta\theta}.
\end{align}
In the subspace
\begin{align}\label{circH}
H=\left\{\rho\in L^2:\ \int_{\T}\rho(r,\theta)\dd \theta=0,\ \text{for a.e. }  r\geq 0\right\},
\end{align}
it can be proven that $\uu$ is diffusion enhancing. 
\begin{theorem}[\cite{CZD19}*{Theorem 1.1}]\label{thm:thirdthm}
Assume that $q\geq 1$. There exist $\eps\in(0,1)$ only depending on $q$ such that the enhanced diffusion rate on $H$ is
\begin{align}\label{eq:decayrate3}
r(\kappa) =\eps \frac{\kappa^{\frac{q}{q+2}}}{1+(q-1)|\ln\kappa|}.
\end{align}
\end{theorem}
Notice that, again, when $q=1$ there is no log-correction. The proof of this theorem is also carried out using hypocoercivity. However,
on the whole space $\R^2$, without a Poincar\'e inequality, it is not so clear that one should obtain an exponential decay estimate
of the type \eqref{eq:endecay}.

\subsection{The main result}
The main result of this paper is that the powers of $\kappa$ appearing in enhanced diffusion rates of 
Theorems \ref{thm:firstthm}--\ref{thm:thirdthm} are sharp, up to  the possible logarithmic corrections.

\begin{theorem}\label{thm:main} The enhanced diffusion rates of 
Theorems \ref{thm:firstthm}, \ref{thm:secondthm}  and \ref{thm:thirdthm} are sharp. Specifically:
\begin{enumerate}
\item[1.] Let $u$ be a time-independent $ C^{n+1}(\Omega)$  shear flow with a \emph{finite} number of critical points, 
where  $n\in \N$ denotes the maximal order of vanishing of $u'$ at the critical points.
There exist initial data $\rho_0\in H$ such that $r(\kappa)\sim\kappa^\frac{n}{n+2}$ is the 
best possible enhanced diffusion rate.
\vspace{2mm}

\item[2.] Let $u$ be a time-independent $C^\alpha$ shear flow. There exist initial data $\rho_0\in H$ 
such that the rate $r(\kappa)\sim\kappa^\frac{\alpha}{\alpha+2}$ is the 
best possible enhanced diffusion rate. In the case of Lipschitz time-independent shear
flows, $r(\kappa)\sim\kappa^{1/3}$  is the 
best possible enhanced diffusion rate. 
\vspace{2mm}

\item[3.] Let $\uu$ be the time-independent  circular flow given by \eqref{def:u} with  $q\geq 1$.
There exist initial data $\rho_0\in H$ such that the enhanced diffusion rates of 
Theorem \ref{thm:thirdthm} cannot be improved, i.e. $r(\kappa)\sim\kappa^\frac{q}{q+2}$ is the 
best possible enhanced diffusion rate.
\end{enumerate}
\end{theorem}

The form of the initial datum $\rho_0$ is fairly simple: in the case of shear flows it is has a non-trivial smooth dependence on $x$, a (at least) Lipschitz dependence on $y$
and it is concentrated near a critical point of maximal order.  In the radial flow case, the datum is concentrated in an annulus around the origin. 
We remark that the proof of sharpness of Theorem \ref{thm:secondthm} does not use the particular form of the shear, rather only its regularity. Therefore, we infer the more general result 
that $\kappa^{1/3}$ is the best possible enhanced
diffusion rate for a time-independent Lipschitz shear flow, and  $\kappa^\frac{\alpha}{\alpha+2}$   is the best possible enhanced
diffusion rate for a time-independent $C^\alpha$ shear flow.
In particular, it also establishes sharpness of the rates of Theorem \ref{thm:firstthm} for $n=1$, simply taking $\alpha=1$.  
\vspace{2mm}

\noindent \textbf{Idea of the proof.} 
To explain the key idea, consider for simplicity \eqref{eq:shear} for the shear $u(y)=y^n$ and $\delta=0$. In principle, this is 
a good approximation near a critical point. To obtain information on the time decay of the energy, we upper bound the cumulative dissipation,
which we identify by means of the Lagrangian fluctuation dissipation relation \eqref{eq:FDR} as the variance of the initial data sampled by backwards noisy trajectories \eqref{eq:backtraj}.
For nice initial conditions $\rho_0$, this variance is dominated  by the behavior of trajectories near the critical points, where they disperse slowest. 
Explicitly writing \eqref{eq:backtraj} for initial data $x=y=0$ (i.e. for particle starting in the 
critical point), we find from  \eqref{eq:backtraj} (see e.g. \eqref{eq:xi}--\eqref{eq:zeta} in what follows) that
\begin{align}
 X_{t,s}(0,0)  &=  (-1)^{n+1}(2\kappa)^{n/2}\int_s^t W_\tau^n \rmd \tau, \\
  Y_{t,s}(0,0)  &= \sqrt{2\kappa} W_s.
\end{align}
Hence, by \eqref{eq:FDR}, the dissipation is controlled for short times near the critical point by the variance
\begin{align}
 \E|X_{t,0}(0,0)|^2  &=  c_n \kappa^{n}t^{n+2}.
\end{align}
In particular, the time-scale in \eqref{eq:decayrate} appears, up to logarithmic corrections, 
when computing the variance of $X_{t,0}$. In other words, a particle
starting in the stationary point $y=0$ takes a time proportional to $1/\kappa^\frac{n}{n+2}$ to reach, in expectation, a distance of
order 1 from it. This effect is solely due to the noise, although the time depends on the nature of the critical point in question: in particular,
it is longer in the case the critical point is flatter.
In the case of radial flows, the situation is analogous, although the proof is more complicated because we cannot solve explicitly 
for the backward trajectories. Nonetheless, sharpness of the rates of Theorem \ref{thm:thirdthm}  can be deduced in this case as well,
up to the log-correction.
\vspace{2mm}

\noindent \textbf{Outline of the paper.} 
The rest of  the paper is devoted to make the heuristic ideas outlined above mathematically sound. 
The next \S\ref{sec:sec2} contains a derivation of a general criterion to prove sharpness of enhanced dissipation
rates, based on the Lagrangian fluctuation-dissipation relation \eqref{eq:FDR}. Section \S \ref{sec:main} contains the proof of Theorem \ref{thm:main} and is divided into subsections covering each part.

\section{Energy dissipation and enhanced diffusion time-scales}\label{sec:sec2}
In this section, we use \eqref{eq:FDR} to properly estimate the dissipation of \eqref{eq:shear} in terms of the initial datum and the
statistics of the backward trajectories \eqref{eq:backtraj}, see Lemma \ref{lem:upperFDR} below. 
Moreover, we prove a simple (but very effective) general criterion 
that relates upper bounds on the dissipation with enhanced diffusion rates (c.f. Lemma \ref{lem:crucial}).

\subsection{The Lagrangian fluctuation dissipation relation (FDR)} \label{secFDR}

For the sake of completeness, we derive the Lagrangian FDR in the setting of shear flows below.  A more general derivation and discussion can be found in \cite{drivas2017lagrangian}, its modification for domains with solid boundary in \cite{drivas2017lagrangian2} and its implications for Rayleigh-B\'{e}nard turbulence in \cite{eyink2018lagrangian}. To begin with, we note that the stochastic trajectories $\bxi_{t,s}(\bx)=( X_{t,s}(x,y),  Y_{t,s}(x,y)  )$, whose (backwards) infinitesimal generator corresponds to the operator 
 $\boldsymbol{u}\cdot \nabla   -\kappa \Delta_\delta $, satisfy \eqref{eq:backtraj}.
 For any $f$ which is spatially $C^2$ and differentiable in time, the backward It\^{o} formula gives
\begin{align*}
\hat{\rmd} f(s, \bxi_{t,s}(\bx))  &= \big( \partial_t f + \boldsymbol{u}\cdot \nabla  f - \kappa \Delta_\delta f \big)|_{(s,\bxi_{t,s}(\bx))} \rmd s\\
&\qquad\qquad\qquad  + \sqrt{2\kappa}\ \hat{\rmd} \begin{pmatrix}B_s\\ W_s\end{pmatrix}\cdot \nabla_\delta f|_{(s,\bxi_{t,s}(\bx))},
\end{align*}
where $\nabla_\delta = (\delta \partial_x , \partial_y)$.
Taking $f= \rho(t,\bx)$ and using the equations of motion, we have upon integration
\be\label{itoform}
\rho(t,\bx) =  \rho_0(\bxi_{t,0}(\bx) ) + \sqrt{2\kappa}\int_0^t \hat{\rmd} \begin{pmatrix}B_s\\ W_s\end{pmatrix}\cdot \nabla_\delta\rho(s,{\bxi_{t,s}(\bx)}).
\ee
Using the fact that the final integral is a backwards martingale, taking the expectation of the above yields the Feynman--Kac formula \eqref{feykac}. Using the identity \eqref{feykac}, squaring \eqref{itoform}, taking expectation and using It\^{o} isometry we find
\begin{align*}
 {2\kappa}\int_0^t \mathbb{E}|\nabla_\delta\rho(s,{\bxi_{t,s}(\bx)})|^2 \rmd s &=\mathbb{E}\left| \rho_0(\bxi_{t,0}(\bx)) - \mathbb{E}\left( \rho_0(\bxi_{t,0}(\bx)) \right)\right|^2\\
 &=: {\rm Var} \left( \rho_0(\bxi_{t,0}(\bx) ) \right).
\end{align*}
Integrating in space over $\Omega$ and using volume preservation of the stochastic flow $\bxi_{t,s}$, we find
\be\label{eq:FDR1}
\kappa \int_0^t  \|\nabla_\delta \rho(s)\|_{L^2}^2 \rmd s =\frac{1}{2} \int_{\Omega} {\rm Var}\left(\rho_0(\bxi_{t,0}(\bx)\right)\rmd \bx.
\ee
The following lemma provides an upper bound on the dissipation that is tailored to the backward trajectories.
\begin{lemma}\label{lem:upperFDR} Let 
$\Omega_0\subseteq \Omega$ be a subdomain of $\Omega$.  Then we have the following bound on the variance
\begin{align}\nonumber
\int_{\Omega_0} {\rm Var}\left(\rho_0(\bxi_{t,0}(\bx)\right)\rmd \bx &\leq  \frac{\| \de_x\rho_0\|^2_{L^\infty(\Omega_0)}}{2}\int_{\Omega_0}\E ^{1,2}\left|X_{t,0}^{(1)}(\bx)- X_{t,0}^{(2)}(\bx)\right|^2\rmd \bx\\
&\qquad +   \frac{\| \de_y\rho_0\|^2_{L^\infty(\Omega_0)}}{2}\int_{\Omega_0}\E ^{1,2}\left|Y_{t,0}^{(1)}(\bx)- Y_{t,0}^{(2)}(\bx)\right|^2\rmd \bx
\end{align}
where $\E ^{1,2}$ represents the expectation over trajectories $(X_{t,0}^{(1)}, Y_{t,0}^{(1)})$ and $(X_{t,0}^{(2)}, Y_{t,0}^{(2)})$ satisfying \eqref{eq:backtraj} with independent (time-reversed) two-dimensional Brownian motions $(B^{(1)}_s,W^{(1)}_s)$ and $(B^{(2)}_s,W^{(2)}_s)$.
\end{lemma}
\begin{proof}
Note that for 
any random variable $\tilde{X}$, ${\rm Var}(\tilde{X}) =\frac{1}{2}\mathbb{E}  |\tilde{X}^{(1)}-\tilde{X}^{(2)}|^2$ where $\tilde{X}^{(1)},
\tilde{X}^{(2)}$ are two independent random variables identically distributed as $\tilde{X}$.  Therefore, the FDR \eqref{eq:FDR1} 
can be expressed as
\begin{align}\nonumber
\int_{\Omega_0} &{\rm Var}\left(\rho_0(\bxi_{t,0}(\bx)\right)\rmd \bx 
= \frac{1}{4}\int_{\Omega_0}\mathbb{E}\left|\rho_0(\bxi_{t,0}^{(1)}(\bx))- \rho_0(\bxi_{t,0}^{(2)}(\bx))\right|^2\rmd \bx\\ \nonumber
&\leq \frac{1}{4}\int_{\Omega_0}\mathbb{E}\left[ \| \de_x\rho_0\|_{L^\infty(\Omega_0)}\left|X_{t,0}^{(1)}(\bx)- X_{t,0}^{(2)}(\bx) \right|  + \| \de_y\rho_0\|_{L^\infty(\Omega_0)}\left|Y_{t,0}^{(1)}(\bx)- Y_{t,0}^{(2)}(\bx) \right| \right]^2\rmd \bx.
\end{align}
The result follows upon applying Young's inequality.
\end{proof}

\subsection{Bound on dissipation constrains enhanced diffusion rate}

To prove enhanced diffusion as in Definition \ref{def:enhanced} one needs to prove a uniform \emph{lower bound} on the dissipation 
by a proper time-scale. In other words, small scales need to be created by mixing in a time proportional to the enhanced diffusion time-scale.
The proof of sharpness of such time-scales relies instead on \emph{upper bounds} on the dissipation.

\begin{lemma}\label{lem:crucial}
Let the velocity field $\uu$ be diffusion enhancing on a subspace $H\subset L^2$ at rate $r(\kappa)$. 
Assume that for some $\rho_0\in H$, some constant $C\geq1$, some continuous increasing function $f:[0,\infty)\to[0,\infty)$ with $f(0)=0$, and some 
increasing function $\xi:(0,\kappa_0]\to (0,1)$, the solution $\rho$ to \eqref{eq:DDE} with initial datum $\rho_0$ satisfies
\begin{align}\label{eq:updis}
\kappa \int_0^t \|\nabla_\delta \rho(s)\|_{L^2}^2 \rmd s\leq f(\xi(\kappa) t)\|\rho_0\|^2_{L^2},\qquad \forall t\leq\frac{1}{C\xi(\kappa)},
\end{align}
for every $\kappa\in(0,\kappa_0]$.
Then
\begin{align}\label{eq:ratefun}
\liminf_{\kappa\to0} \frac{\xi(\kappa)}{r(\kappa)}>0.
\end{align}
\end{lemma} 

\begin{proof}
The starting point of the proof is the energy identity \eqref{eq:enreq}, which in this case reads
\begin{align}\label{eq:energyequation}
\frac12\|\rho(t)\|^2_{L^2}+\kappa \int_0^t \|\nabla_\delta\rho(s)\|^2_{L^2}\dd s=\frac12 \|\rho_0\|^2_{L^2}.
\end{align}
Assume for contradiction that
\begin{align}\label{eq:contrad}
\liminf_{\kappa\to0} \frac{\xi(\kappa)}{r(\kappa)}=0.
\end{align}
Then there exists a subsequence $\kappa_i\to 0$ such that $ \frac{\xi(\kappa_i)}{r(\kappa_i)}\to 0$ as $i\to \infty$.  Abusing notation in what follows, we restrict ourselves to this subsequence and denote $\kappa=\kappa_i$.
In light of \eqref{eq:endecay} and \eqref{eq:updis}, we deduce from \eqref{eq:energyequation} that
\begin{align}\label{eq:contrad2}
1- 2 f(\xi(\kappa) t)\leq C \e^{- r(\kappa)t}, \qquad \forall t\geq0.
\end{align}
Let us now define
\begin{align}
t_\kappa:=\left(\frac{1}{\xi(\kappa) r(\kappa)}\right)^{1/2}.
\end{align}
Thanks to \eqref{eq:contrad}, there exists $0<\kappa_1\leq\kappa_0$ such that
\begin{align}
\frac{1}{r(\kappa)}< t_\kappa< \frac{1}{\xi(\kappa)},\qquad \forall \kappa\in(0,\kappa_1).
\end{align}
Setting $t=t_\kappa$ in \eqref{eq:contrad2}, we find
\begin{align}
1-  2f\left( \left(\frac{\xi(\kappa)}{r(\kappa)}\right)^{1/2} \right)\leq C \exp\left[- \left(\frac{r(\kappa)}{\xi(\kappa)}\right)^{1/2}\right].
\end{align}
Letting $\kappa\to 0$ and using \eqref{eq:contrad}, it is immediate to check that the left-hand side tends to one, while
the right-hand side vanishes. This is a contradiction, and the proof is finished.
\end{proof}

This lemma simply says that a suitable upper bound on the dissipation rate gives a constraint on the enhanced diffusion rate $r(\kappa)$.
If \eqref{eq:updis} is satisfied for some $\xi(\kappa)$ and some initial datum, then the best possible enhanced diffusion rate for $\uu$
is $\xi(\kappa)$.

\begin{remark}
Lemma \ref{lem:crucial} clearly holds for the more general advection-diffusion equation \eqref{eq:DDE} and more general dissipation operators.
All  is needed is an energy equation of the form \eqref{eq:enreq}, with a possibly more general dissipation term.
\end{remark}

\section{Proof of Theorem \ref{thm:main}: Sharpness of enhanced diffusion time-scale}\label{sec:main}
The proofs of each point in the theorem are similar. Therefore, we will carry out in detail the first case, and highlight the main differences for the others.

\subsection{Regular shear flows with critical points}\label{subMain1}
We begin with the proving the first point in Theorem \ref{thm:main}.
Without loss of generality, we can assume that one of the critical points of highest order is at $y=0$, so that, 
\begin{align}
u^{(j)}(0)=0, \qquad \forall  j=1,\ldots, n-1, \quad \text{ and } \quad u^{(n)}(0)\neq 0.
\end{align}
We define the initial datum $\rho_0$ as
\begin{align}\label{eq:theta0}
\rho_0(x,y)=\varphi (y)\sin x ,
\end{align}
where the function $\varphi$ (see Figure \ref{fig:phi}) is given by
\be\label{eq:phi}
\varphi(y)=\begin{cases}
\displaystyle\frac{\kappa^\beta-|y|}{\kappa^\beta}, \quad &y\in [- \kappa^\beta, \kappa^\beta],\\
0, \quad &\text{elsewhere},
\end{cases}
\ee
and 
\be\label{eq:beta}
\beta:=\frac{1}{n+2}.
\ee
Essentially,  $\varphi$ localizes the initial datum near one of the critical points of highest order in a Lipschitz way.

\begin{figure}[h!]
  \includegraphics[width=10cm]{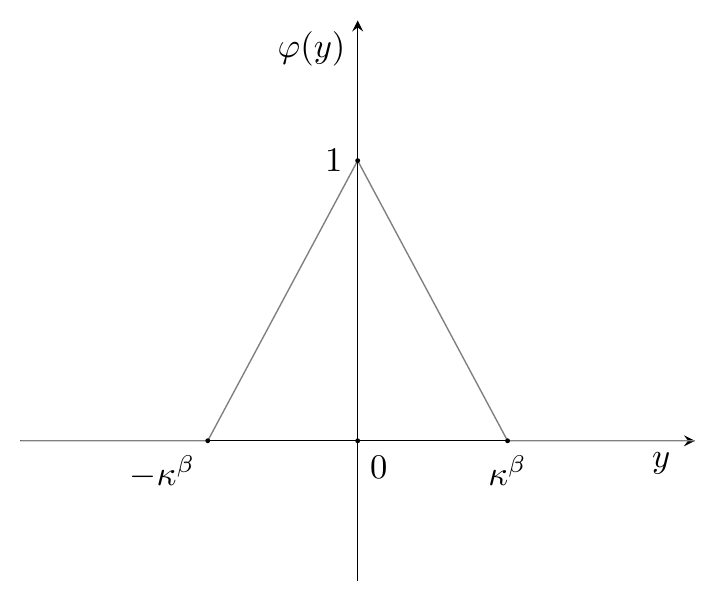}
  \caption{The function $\varphi$}
  \label{fig:phi}
\end{figure}

Since for any $p\geq 1$ we have
\be
\int_{D}|\varphi(y)|^p\dd y=\frac{2\kappa^{\beta}}{p+1}, \qquad \|\de_y\varphi\|^2_{L^\infty}=\frac{1}{\kappa^{2\beta}},
\ee
we then have for some constant $c_{p}>0$ independent of $\kappa$ that 
\begin{align}\label{eq:boundstheta0}
 \|\rho_0\|^2_{L^p}=c_p\kappa^{2\beta/p},\qquad
\|\de_x \rho_0\|^2_{L^\infty}= 1,\qquad
\|\de_y \rho_0\|^2_{L^\infty}=\frac{1}{\kappa^{2\beta}}.
\end{align}
The idea is to apply Lemma \ref{lem:upperFDR}  to obtain a suitable upper bound on the energy dissipation rate. Define
\be
\Omega_{\kappa}=\T\times I_{\kappa},\qquad I_{\kappa}=[-2\kappa^{\beta},2\kappa^\beta],\qquad \Omega'_{\kappa}=\T\times I^c_{\kappa}.
\ee
We now split
\begin{align}\label{eq:split}
\int_{\Omega} {\rm Var}\left(\rho_0(\bxi_{t,0}(\bx)\right)\rmd \bx
=\int_{\Omega_\kappa} {\rm Var}\left(\rho_0(\bxi_{t,0}(\bx))\right)\rmd \bx+\int_{\Omega'_\kappa} {\rm Var}\left(\rho_0(\bxi_{t,0}(\bx))\right)\rmd \bx,
\end{align}
and estimate both terms from above separately.  For the shear flow, we note that the trajectories simplify 
\begin{align}
 X_{t,s}(x,y)  &= x  +\delta \sqrt{2\kappa} B_s  -\int_s^t u( y + \sqrt{2\kappa} W_\tau) \rmd \tau, \label{eq:xi}\\
 Y_{t,s}(x,y)  &= y + \sqrt{2\kappa} W_s. \label{eq:zeta}
\end{align}

\medskip

\noindent \textbf{Contribution of the region $\Omega_{\kappa}$.}
Using Lemma \ref{lem:upperFDR} with $\Omega_0=\Omega_\kappa$,  \eqref{eq:boundstheta0}, and referring to \eqref{eq:xi}--\eqref{eq:zeta}, 
\begin{align}
\int_{\Omega_\kappa} {\rm Var}\left(\rho_0(\bxi_{t,0}(\bx))\right)\rmd \bx 
&\leq \nonumber
\frac{\| \de_x\rho_0\|^2_{L^\infty}}{2}\int_{\Omega_\kappa}\E ^{1,2}\left|X_{t,0}^{(1)}(\bx)- X_{t,0}^{(2)}(\bx)\right|^2\rmd \bx\\
&\qquad +   \frac{\| \de_y\rho_0\|^2_{L^\infty}}{2}\int_{\Omega_\kappa}\E ^{1,2}\left|Y_{t,0}^{(1)}(\bx)- Y_{t,0}^{(2)}(\bx)\right|^2\rmd \bx\notag\\
&\lesssim
\int_{\Omega_\kappa}\E ^{1,2}\left|X_{t,0}^{(1)}(\bx)- X_{t,0}^{(2)}(\bx)\right|^2\rmd \bx   \notag\\
&\qquad 
+\frac{1}{\kappa^{2\beta}}\int_{\Omega_\kappa}\E ^{1,2}\left|Y_{t,0}^{(1)}(\bx)- Y_{t,0}^{(2)}(\bx)\right|^2\rmd \bx. \nonumber
\end{align}
To estimate the first term, we compute
\begin{align}\nonumber
\E ^{1,2}\left|X_{t,0}^{(1)}(\bx)- X_{t,0}^{(2)}(\bx)\right|^2 &\lesssim \kappa\E ^{1,2}|B^{(1)}_0-B^{(2)}_0|^2  \\
&\qquad +\E ^{1,2}  \left| \int_0^t (u( y +  \sqrt{2\kappa} W_\tau^{(1)})- u( y +  \sqrt{2\kappa} W_\tau^{(2)})) \rmd \tau\right|^2. \label{eq:fasd}
\end{align}
The first term is easily estimated as
\begin{align}\label{eq:bound1}
\kappa\E ^{1,2}|B^{(1)}_0-B^{(2)}_0|^2 \lesssim \kappa t.
\end{align}
For the second term, we preliminary notice that if $y\in [-2\kappa^{\beta},2\kappa^\beta]$, we have
\begin{align}
|u( y + \sqrt{2\kappa} W_\tau^{(1)})-u(0)|\lesssim |y|^n +\kappa^{n/2}| W_\tau^{(1)}|^n\lesssim \kappa^\frac{n}{n+2}+\kappa^{n/2}| W_\tau^{(1)}|^n.
\end{align}
Therefore,  if $y\in [-2\kappa^{\beta},2\kappa^\beta]$, we have
\begin{align}\label{eq:onlydiff}
\E ^{1,2}  &\left| \int_0^t (u( y + \sqrt{2\kappa} W_\tau^{(1)})- u( y + \sqrt{2\kappa} W_\tau^{(2)})) \rmd \tau\right|^2\notag\\
&\lesssim  t\int_0^t\E ^{1,2}  \left[ \left| (u( y + \sqrt{2\kappa} W_\tau^{(1)})-u(0)\right|^2 +\left|u(0)-u( y + \sqrt{2\kappa} W_\tau^{(2)}))\right|^2 \right]\rmd \tau\notag\\
&\lesssim  t\int_0^t\E ^{1,2}  \left[ \kappa^\frac{2n}{n+2}+\kappa^{n}| W_\tau^{(1)}|^{2n}+\kappa^{n}| W_\tau^{(2)}|^{2n} \right]\rmd \tau.
\end{align} 
Using that for a standard time-reversed one-dimensional Brownian motion there holds
\begin{align}\label{eq:Brownmom}
\E  W_\tau^{2n}= \frac{(2n)!}{n! 2^n} (t-\tau)^n, \quad \tau\leq t,
\end{align}
 we end up with
\begin{align}\label{eq:bound2}
\E ^{1,2}  &\left| \int_0^t (u( y + \sqrt{2\kappa} W_\tau^{(1)})- u( y + \sqrt{2\kappa} W_\tau^{(2)})) \rmd \tau\right|^2
\lesssim  \left(\kappa^\frac{n}{n+2}t\right)^2+\left(\kappa^\frac{n}{n+2}t\right)^{n+2}.
\end{align}
Consequently, appealing to \eqref{eq:bound1}, \eqref{eq:bound2} and \eqref{eq:boundstheta0} and the fact that
$|\Omega_\kappa|=8\pi  \kappa^\beta$ is comparable to $\|\rho_0\|^2_{L^2}$, we deduce that
\begin{align}\label{eq:fasd2}
\int_{\Omega_\kappa}\E ^{1,2}\left|X_{t,0}^{(1)}(\bx)- X_{t,0}^{(2)}(\bx)\right|^2\dd \bx \lesssim 
\|\rho_0\|^2_{L^2}\left( \kappa t+\left(\kappa^\frac{n}{n+2}t\right)^2+\left(\kappa^\frac{n}{n+2}t\right)^{n+2}\right).
\end{align}
For the second term in \eqref{eq:fasd}, we recall \eqref{eq:beta}  and obtain
\be
\frac{1}{\kappa^{2\beta}} \int_{\Omega_\kappa}  \E ^{1,2} | Y^{(1)}_{t,0}(\bx) - Y^{(2)}_{t,0}(\bx) |^2 \dd\bx
\lesssim \kappa^{1-2\beta} |\Omega_\kappa| \E ^{1,2}  \left| W_0^{(1)}- W_0^{(2)}\right|^2
\lesssim\|\rho_0\|^2_{L^2}\kappa^\frac{n}{n+2} t.
\ee
Hence, since $\kappa\leq \kappa^\frac{n}{n+2}$, we conclude that
\begin{align}\label{eq:boundVar1}
\int_{\Omega_\kappa} {\rm Var}\left(\rho_0(\bxi_{t,0}(\bx))\right)\rmd \bx 
\lesssim \|\rho_0\|^2_{L^2}\left(\kappa^\frac{n}{n+2} t+\left(\kappa^\frac{n}{n+2}t\right)^2+\left(\kappa^\frac{n}{n+2}t\right)^{n+2}\right).
\end{align}
\medskip

\noindent \textbf{Contribution of the region $\Omega_{\kappa}'$.}
In order to give an upper bound of the second term in \eqref{eq:split}, we can first argue point-wise for $y\in I^c_{\kappa}$.
Chebyshev inequality and \eqref{eq:Brownmom} imply that
\be
\mathbb{P}\left[\sqrt{2\kappa}|W_0|>|y|-\kappa^\beta\right]\leq \frac{4\kappa^{2}}{(|y|-\kappa^\beta)^{4}}\E |W_0|^{4}
= 12\frac{\kappa^{2} t^2}{(|y|-\kappa^\beta)^{4}}.
\ee
Therefore,
\be
\mathbb{P}\left[ Y_{t,0}(x,y)\in \spt(\varphi)\right]\leq 12\frac{\kappa^{2}t^2}{(|y|-\kappa^\beta)^{4}}, \qquad \forall (x,y)\in \Omega'_\kappa,
\ee
so that 
\be\label{eq:boundell}
\mathbb{P}\left[ \rho_0(\bxi_{t,0}(\bx))\neq 0\right]\leq 12\frac{\kappa^{2}t^2}{(|y|-\kappa^\beta)^{4}}, \qquad \forall \bx\in \Omega'_\kappa.
\ee
Thus, applying H\"older inequality and using \eqref{eq:boundstheta0} we deduce that
\begin{align}
\int_{\Omega'_\kappa}\E |\rho_0(\bxi_{t,0}(\bx))|^2\rmd \bx
&= \int_{\Omega'_\kappa}\E \left[\mathbbm{1}_{\left[ \rho_0(\bxi_{t,0}(\bx))\neq 0\right]} |\rho_0(\bxi_{t,0}(\bx))|^2\right]\rmd \bx\notag\\
&\leq \left( \int_{\Omega'_\kappa}\E  |\rho_0(\bxi_{t,0}(\bx))|^4\rmd \bx\right)^{1/2}\left( \int_{\Omega'_\kappa}\E \left[\mathbbm{1}_{\left[ \rho_0(\bxi_{t,0}(\bx))\neq 0\right]}\right] \rmd \bx\right)^{1/2}\notag\\
&\lesssim \|\rho_0\|^2_{L^4}\kappa t\left( \int_{\Omega'_\kappa}\frac{1}{(|y|-\kappa^\beta)^{4}} \rmd x\dd y\right)^{1/2}\notag\\
&\lesssim \|\rho_0\|^2_{L^4}\kappa t\left(  \int_{2\kappa^\beta}^\infty\frac{1}{(y-\kappa^\beta)^{4}} \rmd y\right)^{1/2}\notag
\lesssim \|\rho_0\|^2_{L^2}\kappa^\frac{n}{n+2} t.
\end{align}
Hence, we are able to conclude that
\begin{align}\label{eq:boundVar2}
\int_{\Omega'_\kappa} {\rm Var}\left(\rho_0(\bxi_{t,0}(\bx))\right)\rmd \bx\lesssim \int_{\Omega'_\kappa}\E |\rho_0(\bxi_{t,0}(\bx))|^2\rmd \bx
\lesssim  \|\rho_0\|^2_{L^2} \kappa^\frac{n}{n+2} t.
\end{align}
Thus, \eqref{eq:boundVar1} together with \eqref{eq:boundVar2} yields
\begin{align}
\int_{\Omega} {\rm Var}\left(\rho_0(\bxi_{t,0}(\bx)\right)\rmd \bx\lesssim
 \|\rho_0\|^2_{L^2}\left(\kappa^\frac{n}{n+2} t+\left(\kappa^\frac{n}{n+2}t\right)^2+\left(\kappa^\frac{n}{n+2}t\right)^{n+2}\right),
\end{align}
which, by the FDR \eqref{eq:FDR1}, it implies the upper bound
\be\label{eq:finalbou}
\kappa \int_0^t  \|\nabla \rho(s)\|_{L^2}^2 \rmd s \lesssim
 \|\rho_0\|^2_{L^2}\left(\kappa^\frac{n}{n+2} t+\left(\kappa^\frac{n}{n+2}t\right)^2+\left(\kappa^\frac{n}{n+2}t\right)^{n+2}\right).
\ee
We are now in the position of applying Lemma \ref{lem:crucial} and conclude the proof of Theorem \ref{thm:main} in the first case of regular 
shear flows with critical points (namely, sharpness of Theorem \ref{thm:firstthm} up to the log correction, in the case $n\geq 2$).

\subsection{H\"older and Lipschitz shear flows}
In this section, we prove the second point in Theorem \ref{thm:main}.
Instead of focusing on the specific case of the Weierstrass function \eqref{eq:Wei}, we assume in this section that 
the shear flow $u$ is either (uniformly) H\"older or Lipschitz continuous. Namely, assume
 that for some $\alpha\in (0,1]$, the shear flow driving \eqref{eq:shear} satisfies the bound
\be
|u(y)- u(y')|\leq c|y-y'|^\alpha, \qquad \forall y,y'\in D,
\ee
for some $c>0$. Now, define the initial datum as in \eqref{eq:theta0}, where it does not matter where $\varphi$ is centered. As for the parameter
$\beta$ in \eqref{eq:beta}, simply replace $n$ by $\alpha$, so that
\begin{align}\label{eq:beta2}
\beta:=\frac{1}{\alpha+2}.
\end{align}
Taking the domain-splitting as in \eqref{eq:split} the only difference relies on treating the term \eqref{eq:onlydiff}. In this case,
using Jensen's inequality, we find
\begin{align*}
\E ^{1,2}  &\left| \int_0^t (u( y + \sqrt{2\kappa} W_\tau^{(1)})- u( y + \sqrt{2\kappa} W_\tau^{(2)})) \rmd \tau\right|^2\\
&\qquad \lesssim t \int_0^t \E ^{1,2} |u( y + \sqrt{2\kappa} W_\tau^{(1)})- u( y + \sqrt{2\kappa} W_\tau^{(2)})|^2\rmd \tau \notag\\
&\qquad \lesssim \kappa^\alpha t\int_0^t \E ^{1,2} | W_\tau^{(1)}-W_\tau^{(2)}|^{2\alpha}\rmd \tau\lesssim \kappa^\alpha t\int_0^t \left(\E ^{1,2} | W_\tau^{(1)}-W_\tau^{(2)}|^{2}\right)^{\alpha}\rmd \tau\lesssim  \kappa^\alpha t\int_0^t \tau^{\alpha}\rmd \tau.
\end{align*} 
Hence \eqref{eq:bound2} is replaced by
\begin{align}\label{eq:bound2bis}
\E ^{1,2}  &\left| \int_0^t (u( y + \sqrt{2\kappa} W_\tau^{(1)})- u( y + \sqrt{2\kappa} W_\tau^{(2)})) \rmd \tau\right|^2
\lesssim  \kappa^\frac{\alpha}{\alpha+2}t,
\end{align}
and \eqref{eq:boundVar1} becomes
\begin{align}\label{eq:boundVar1bis}
\int_{\Omega_\kappa} {\rm Var}\left(\rho_0(\bxi_{t,0}(\bx))\right)\rmd \bx 
\lesssim \|\rho_0\|^2_{L^2}\kappa^\frac{\alpha}{\alpha+2} t.
\end{align}
Concerning \eqref{eq:boundVar2}, there is no difference, besides replacing $n$ with $\alpha$. Therefore
\eqref{eq:finalbou} takes the form
\be\label{eq:finalbou2}
\kappa \int_0^t  \|\nabla \rho(s)\|_{L^2}^2\rmd s \lesssim
 \|\rho_0\|^2_{L^2}\kappa^\frac{\alpha}{\alpha+2} t.
\ee
We can therefore appeal to  Lemma \ref{lem:crucial} and conclude the proof of Theorem \ref{thm:main} in the second case and in the first for 
$n=1$.

\subsection{Circular flows}\label{sec:circ}
In this section, we prove the final point in Theorem \ref{thm:main}.
To proceed, we set
\begin{align}\label{eq:beta3}
\beta:=\frac{1}{q+2},
\end{align}
and fix an initial datum localized near $r=0$, but \emph{away from it}, so $\varphi$ is the same as in \eqref{eq:phi} and
\begin{align}
\rho_0(r,\theta)=\varphi(r-3\kappa^\beta)\sin \theta.
\end{align}
Notice that $\rho_0$ is supported in the $r$-strip $[2\kappa^\beta, 4\kappa^\beta]$.
Take
\be
\Omega_{\kappa}= I_{\kappa}\times \T,\qquad I_{\kappa}=[\kappa^\beta,5\kappa^\beta],\qquad \Omega'_{\kappa}= I^c_{\kappa}\times\T.
\ee
Moreover, since
\begin{align}
\int_0^\infty |\varphi(r)|^p r\dd r=c_p\kappa^{2\beta},
\end{align}
we have
\begin{align}\label{eq:boundstheta0bis}
 \|\rho_0\|^2_{L^p}=c_p\kappa^{4\beta/p},\qquad
\|\de_\theta \rho_0\|^2_{L^\infty}= 1,\qquad
\|\de_r \rho_0\|^2_{L^\infty}=\frac{1}{\kappa^{2\beta}}.
\end{align}
Recall now that the advection diffusion equation written in polar coordinates \eqref{eq:cauchycirc}, \eqref{laplacepolar} reads
\begin{align}\label{polarAD}
\de_t \rho+r^q \de_\theta \rho =\kappa \left(\de_{rr}+\frac{1}{r} \de_r +\frac{1}{r^2}\de_{\theta\theta}\right)\rho.
\end{align}
For the representation formulae for solutions of \eqref{polarAD} and their dissipation, we introduce
the   stochastic characteristics $\Theta_{t,s}(r,\theta), R_{t,s}(r,\theta)$ which satisfy the following backward It\^{o} SDEs
\begin{align}
\hat{\dd} \Theta_{t,s}= R_{t,s}^q\dd s+\frac{\sqrt{2\kappa}}{R_{t,s}}\hat{\rmd} B_s  ,\qquad \Theta_{t,t}=\theta,\label{eq:Theta}\\
\hat{\dd} R_{t,s} =-\frac{\kappa}{R_{t,s}}\dd s+\sqrt{2\kappa}\ \hat{\rmd} W_s,\qquad R_{t,t}=r.\label{eq:R}
\end{align}
We remark that with these characteristics, the Feynman-Kac formula reads
\be
\rho_t(r,\theta) = \mathbb{E}\left[ \rho_0(R_{t,0}(r,\theta), \Theta_{t,0}(r,\theta))\right].
\ee
We will require the following lemma on radial and angular  particle dispersion.

\begin{lemma}
Let $\Theta_{t,s}(r,\theta), R_{t,s}(r,\theta)$ solve \eqref{eq:Theta} and \eqref{eq:R} respectively.  Then, we have
\begin{align}\label{eq:good1}
\E ^{1,2}\left|R_{t,0}^{(1)}(r,\theta)- R_{t,0}^{(2)}(r,\theta)\right|^2
&\lesssim \kappa(t-s),
\end{align}
and 
\begin{align}\label{eq:THETA2}
\E ^{1,2}\left|\Theta_{t,0}^{(1)}(r,\theta)- \Theta_{t,0}^{(2)}(r,\theta)\right|^2
\lesssim  t^2\left[r^{2	q }+\kappa^q t^q\right]+ \ln\left(\frac{1}{1-4\kappa t/r^2}\right) ,\qquad \forall t< \frac{r^2}{4\kappa}.
\end{align}
\end{lemma}
\begin{proof}
To prove \eqref{eq:good1}, we compute  from \eqref{eq:R} that
\begin{align}
\hat{\dd} (R_{t,s}-r)^2 &=-\left(2\kappa\frac{R_{t,s}-r}{R_{t,s}}+2\kappa\right)\dd s+2\sqrt{2\kappa}(R_{t,s}-r) \hat{\rmd} W_s\notag\\ 
&=-\left(4\kappa-2\kappa\frac{r}{R_{t,s}}\right)\dd s+2\sqrt{2\kappa}(R_{t,s}-r) \hat{\rmd} W_s. \label{eq:R2}
\end{align}
Integrating and taking expectations (using the fact aht $R_{t,\tau}\geq 0$) we find
\begin{align}\label{eq:Rminusr}
\E(R_{t,s}-r)^2 =4\kappa(t-s)-2\kappa\int_s^t\E\frac{r}{R_{t,\tau}}\dd \tau\leq 4\kappa(t-s).
\end{align}
In particular we have the bound \eqref{eq:good1} on the radial dispersion of trajectories.
To obtain the estimate \eqref{eq:THETA2}, we note that from \eqref{eq:Theta} we have 
\begin{align}
\Theta_{t,0}= \theta -\int_0^t R_{t,s}^q\dd s-\sqrt{2\kappa}\int_0^t\frac{1}{R_{t,s}}\hat{\rmd} B_s.
\end{align}
It follows by It\^{o} isometry and Cauchy-Schwarz that
\begin{align}\label{eq:THETA1}
\E|\Theta_{t,0}-\theta|^2\lesssim t\int_0^t \E|R_{t,s}|^{2q}\dd s +\kappa\int_0^t \E R_{t,s}^{-2}\dd s.
\end{align}
We now bound these two terms individually.  For the first, note that  
\begin{align}\label{eq:pmom}
\hat{\dd} R_{t,s}^p =-\kappa p^2 R_{t,s}^{p-2}\dd s+p\sqrt{2\kappa}\ R_{t,s}^{p-1}\hat{\rmd} W_s,
\end{align}
so that for any real $p$
\begin{align}
 R_{t,s}^p =r^p+\kappa p^2 \int_s^t R_{t,\tau}^{p-2}\dd \tau
 -p\sqrt{2\kappa}\int_s^t R_{t,\tau}^{p-1}\hat{\rmd} W_\tau.
\end{align}
Hence the second moment satisfies
\begin{align}\label{eq:mombase}
\E |R_{t,s}|^2 =r^2+4\kappa (t-s).
\end{align}
Using \eqref{eq:mombase}, we iterate \eqref{eq:pmom} to obtain for any $\ell\in \N$ that
\begin{align}\label{eq:highmoms}
\E |R_{t,s}|^{2\ell} \lesssim r^{2\ell}+\kappa^\ell (t-s)^\ell.
\end{align}
To estimate the first term in \eqref{eq:THETA1}, we use \eqref{eq:highmoms} with $\ell =\lceil q\rceil \geq q$ and  Jensen's inequality
together with the simple bound  $(x+y)^\alpha\lesssim x^\alpha+y^\alpha$ to find
\begin{align}
\int_0^t \E|R_{t,s}|^{2q}\dd s=\int_0^t \E\left[|R_{t,s}|^{2\ell}\right]^{q/\ell}\dd s &\leq\int_0^t \left[\E|R_{t,s}|^{2\ell}\right]^{q/\ell}\dd s
\lesssim\int_0^t \left[r^{2\ell}+\kappa^\ell (t-s)^\ell\right]^{q/\ell}\dd s\notag\\
&\lesssim\int_0^t \left[r^{2q}+\kappa^q (t-s)^q\right]\dd s\lesssim t\left[r^{2	q }+\kappa^q t^q\right] .
\end{align}
To estimate the second term in \eqref{eq:THETA1}, we note that from \eqref{eq:pmom} with $p=-2$ we have
\begin{align}\label{eq:inverse}
\frac{\dd}{\dd s} \E R_{t,s}^{-2} =-4\kappa  \E R_{t,s}^{-4}\leq-4\kappa  \left[\E R_{t,s}^{-2}\right]^2 ,\qquad \E R_{t,t}^{-2}=\frac{1}{r^2}
\end{align}
using Jensen's inequality.
Calling $y(s)=\E R_{t,s}^{-2}$, we have the \emph{backward} ODE
\begin{align}
y'\leq -4\kappa y^2.
\end{align}
Provided $4\kappa(t-s)< r^2$ we have upon integrating the above differential inequality
\begin{align}
 \E R_{t,s}^{-2} \leq \frac{1}{r^2-4\kappa(t-s)}.
\end{align}
We use this to estimate the contribution to \eqref{eq:THETA1} as follows
\begin{align}\label{eq:negmom}
\kappa\int_0^t \E R_{t,s}^{-2}\dd s\leq \kappa\int_0^t\frac{\dd s}{r^2-4\kappa(t-s)}=  \ln\left(\frac{1}{1-4\kappa t/r^2}\right) ,\qquad \forall t< \frac{r^2}{4\kappa} .
\end{align}
Thus, from \eqref{eq:THETA1} and the above, we deduce that
\begin{align}
\E|\Theta_{t,0}-\theta|^2
\lesssim  t^2\left[r^{2	q }+\kappa^q t^q\right]+ \ln\left(\frac{1}{1-4\kappa t/r^2}\right) ,\qquad \forall t< \frac{r^2}{4\kappa}.
\end{align}
This establishes the angular dispersion estimate \eqref{eq:THETA2} and completes the proof.
\end{proof}

We now continue with the proof of the third part of Theorem \ref{thm:main}. First note that, as in Lemma \ref{lem:upperFDR}, for any  $\Omega_0\subset \R^2$ we have
\begin{align}\nonumber
\int_{\Omega_0} {\rm Var}\left(\rho_0(\bxi_{t,0}(\bx)\right)\rmd \bx 
&= \frac{1}{4}\int_{\Omega_0}\mathbb{E}\left|\rho_0(\bxi_{t,0}^{(1)}(\bx))- \rho_0(\bxi_{t,0}^{(2)}(\bx))\right|^2\rmd \bx\\ \nonumber
& \leq  \frac{\| \de_r\rho_0\|^2_{L^\infty(\Omega_0)}}{2}\int_{\Omega_0}\E ^{1,2}\left|R_{t,0}^{(1)}(r,\theta)- R_{t,0}^{(2)}(r,\theta)\right|^2r\rmd r\dd\theta\nonumber\\
&\quad +   \frac{\| \de_\theta\rho_0\|^2_{L^\infty(\Omega_0)}}{2}\int_{\Omega_0}\E ^{1,2}\left|\Theta_{t,0}^{(1)}(r,\theta)- \Theta_{t,0}^{(2)}(r,\theta)\right|^2r\rmd r \dd\theta.
\end{align}
Recall that we divide the domain $\mathbb{R}\times \mathbb{T}$ into $\Omega_{\kappa}= I_{\kappa}\times \T$ and $\Omega'_{\kappa}= I^c_{\kappa}\times\T$ with $I_{\kappa}=[\kappa^\beta,5\kappa^\beta]$.
\medskip

\noindent \textbf{Contribution of the region $\Omega_{\kappa}$.}
We treat the terms involving $R$ and $\Theta$ separately.  
For the term involving radial dispersion, using \eqref{eq:good1} and  \eqref{eq:boundstheta0bis} we have
\begin{align}
&\int_{\Omega_\kappa}\E ^{1,2}\left|R_{t,0}^{(1)}(r,\theta)- R_{t,0}^{(2)}(r,\theta)\right|^2r\rmd r\dd\theta
\lesssim \kappa^{2\beta}\kappa t \lesssim \|\rho_0\|^2_{L^2} \kappa t.
\end{align}
To estimate the term involving $\Theta$, we use \eqref{eq:THETA2}. Notice that since in this case $r\geq \kappa^{\beta}$, 
the restriction $t< \frac{r^2}{\kappa}$ is in particular satisfied if 
\begin{align}\label{eq:constraint}
t<\frac{1}{4\kappa^\frac{q}{q+2}}.
\end{align}
For these times we obtain
\begin{align}
\int_{\Omega_\kappa}\E ^{1,2}\left|\Theta_{t,0}^{(1)}(r,\theta)- \Theta_{t,0}^{(2)}(r,\theta)\right|^2r\rmd r \dd\theta
&\lesssim\int_{\Omega_\kappa}\left[t^2\left[r^{2	q }+\kappa^q t^q\right]+ \ln\left(\frac{1}{1-4\kappa t/r^2}\right)\right] r\rmd r \dd\theta\notag\\
&\lesssim \kappa^{2\beta}\left[ \left(\kappa^{\beta q}t\right)^2 +\left(\kappa^\frac{q}{q+2}t\right)^{q+2}\right]\notag\\
&\quad+2\kappa t\ln\left(25\kappa^{2\beta}-4\kappa t\right)-\frac{25\kappa^{2\beta}}{2}\ln\left(1-\frac{4\kappa t}{25\kappa^{2\beta}}\right)\notag\\
&\quad-2\kappa t\ln\left(\kappa^{2\beta}-4\kappa t\right)+\frac{\kappa^{2\beta}}{2}\ln\left(1-\frac{4\kappa t}{\kappa^{2\beta}}\right)\notag\\
&\lesssim  \kappa^{2\beta}\left[ \left(\kappa^\frac{q}{q+2}t\right)^2 +\left(\kappa^\frac{q}{q+2}t\right)^{q+2}\right]\notag\\
&\quad +\kappa^{2\beta}\left[ \kappa^\frac{q}{q+2} t\left|\ln\left(\frac{25-4\kappa^\frac{q}{q+2} t}{1-4\kappa^\frac{q}{q+2} t}\right)\right|+\left|\ln\left(1-\frac{4\kappa^\frac{q}{q+2} t}{25}\right)\right|\right].\nonumber
\end{align}
In the end, in view of \eqref{eq:constraint} and \eqref{eq:boundstheta0bis}, we have
\begin{align}
\int_{\Omega_\kappa}\E ^{1,2}\left|\Theta_{t,0}^{(1)}(r,\theta)- \Theta_{t,0}^{(2)}(r,\theta)\right|^2r\rmd r \dd\theta
&\lesssim  \|\rho_0\|^2_{L^2}\left[ \left(t\kappa^\frac{q}{q+2}\right)^2 +\left(\kappa^\frac{q}{q+2}t\right)^{q+2}\right]\notag\\
& +\|\rho_0\|^2_{L^2}\left[ \kappa^\frac{q}{q+2} t\left|\ln\left(\frac{25-4\kappa^\frac{q}{q+2} t}{1-4\kappa^\frac{q}{q+2} t}\right)\right|+\left|\ln\left(1-\frac{4\kappa^\frac{q}{q+2} t}{25}\right)\right|\right]. \nonumber
\end{align}
Hence,
\begin{align}\nonumber
\int_{\Omega_\kappa} {\rm Var}\left(\rho_0(\bxi_{t,0}(\bx)\right)\rmd \bx \lesssim
\|\rho_0\|^2_{L^2}\bigg[&\kappa^\frac{q}{q+2} t+\left(t\kappa^\frac{q}{q+2}\right)^2 +\left(\kappa^\frac{q}{q+2}t\right)^{q+2}\notag\\
&+ \kappa^\frac{q}{q+2} t\left|\ln\left(\frac{25-4\kappa^\frac{q}{q+2} t}{1-4\kappa^\frac{q}{q+2} t}\right)\right|+\left|\ln\left(1-\frac{4\kappa^\frac{q}{q+2} t}{25}\right)\right|\bigg],
\quad\label{eq:intkappa}
\end{align}
for all $ t<\kappa^{-\frac{q}{q+2}}/4$.
\medskip

\noindent \textbf{Contribution of the region $\Omega_{\kappa}'$.}
Now we need to estimate two cases when $r\in I^c_{\kappa}$, and essentially we need to know
when the trajectory $R_{t,0}(r,\theta)\in [2\kappa^\beta,4\kappa^\beta]$. 
The first case is when $r\in(0,\kappa^\beta)$. In this case, \eqref{eq:Rminusr} implies 
\be
\mathbb{P}\left[R_{t,0}(r,\theta)>2\kappa^\beta\right]=\mathbb{P}\left[R_{t,0}(r,\theta)-r>2\kappa^\beta-r\right]\leq \frac{\E |R_{t,0}(r,\theta)-r|^{2}}{(2\kappa^{\beta}-r)^2}\lesssim \frac{\kappa t}{\kappa^{2\beta}}\lesssim \kappa^\frac{q}{q+2} t.
\ee
Therefore,
\be
\mathbb{P}\left[ R_{t,0}(r,\theta)\in \spt(\varphi)\right]\lesssim \kappa^\frac{q}{q+2} t, \qquad \forall r\in(0,\kappa^\beta).
\ee
The second case is when $r\in(5\kappa^\beta,\infty)$. Then
\begin{align}
\mathbb{P}\left[R_{t,0}(r,\theta)<4\kappa^\beta\right]
&= \mathbb{P}\left[r-R_{t,0}(r,\theta)>r-4\kappa^\beta\right]
=\mathbb{P}\left[-\int_0^t\frac{\kappa}{R_{t,s}}\dd s+\sqrt{2\kappa}\ W_0>r-4\kappa^\beta\right]\notag\\
&\leq\mathbb{P}\left[\sqrt{2\kappa}\ W_0>r-4\kappa^\beta\right]
\leq \frac{4\kappa^{2}}{(r-4\kappa^\beta)^{4}}\E |W_0|^{4}\lesssim \frac{\kappa^{2}t^2}{(r-4\kappa^\beta)^{4}}.
\end{align}
Therefore,
\be
\mathbb{P}\left[ R_{t,0}(r,\theta)\in \spt(\varphi)\right]\lesssim  \frac{\kappa^{2}t^2}{(r-4\kappa^\beta)^{4}}, \qquad \forall r\in(5\kappa^\beta,\infty).
\ee
By writing
\begin{align}
\Omega'_\kappa=\Omega'_{\kappa,1}\cup\Omega'_{\kappa,2}, \qquad \Omega'_{\kappa,1}= (0,\kappa^\beta)\times\T,
\qquad \Omega'_{\kappa,2}= (5\kappa^\beta,\infty)\times\T,
\end{align}
we split
\begin{align}
\int_{\Omega'_\kappa}\E |\rho_0(\bxi_{t,0}(\bx))|^2\rmd \bx=\int_{\Omega'_{\kappa,1}}\E |\rho_0(\bxi_{t,0}(\bx))|^2\rmd \bx
+\int_{\Omega'_{\kappa,2}}\E |\rho_0(\bxi_{t,0}(\bx))|^2\rmd \bx \nonumber
\end{align}
 and argue as we did in \S \ref{subMain1}.  Namely, we obtain the estimate
\begin{align}
\int_{\Omega'_{\kappa,1}}\E |\rho_0(\bxi_{t,0}(\bx))|^2\rmd \bx
&=\int_{\Omega'_{\kappa,1}}\E \left[\mathbbm{1}_{\left[ \rho_0(\bxi_{t,0}(\bx))\neq 0\right]} |\rho_0(\bxi_{t,0}(\bx))|^2\right]\rmd \bx\notag\\
&\leq \left(\int_{\Omega'_{\kappa,1}}\E  |\rho_0(\bxi_{t,0}(\bx))|^4\rmd \bx\right)^{1/2}\left( \int_{\Omega'_{\kappa,1}}\E \left[\mathbbm{1}_{\left[ \rho_0(\bxi_{t,0}(\bx))\neq 0\right]}\right] \rmd \bx\right)^{1/2}\notag\\
&\lesssim \|\rho_0\|^2_{L^4}\left( \kappa^{2\beta}\kappa^\frac{q}{q+2} t \right)^{1/2}\lesssim \|\rho_0\|^2_{L^2}\left(\kappa^\frac{q}{q+2} t \right)^{1/2},\nonumber
\end{align}
and additionally
\begin{align}
\int_{\Omega'_{\kappa,2}}\E |\rho_0(\bxi_{t,0}(\bx))|^2\rmd \bx
&\leq \left( \int_{\Omega'_{\kappa,2}}\E  |\rho_0(\bxi_{t,0}(\bx))|^4\rmd \bx\right)^{1/2}\left( \int_{\Omega'_{\kappa,2}}\E \left[\mathbbm{1}_{\left[ \rho_0(\bxi_{t,0}(\bx))\neq 0\right]}\right] \rmd \bx\right)^{1/2}\notag\\
&\lesssim \|\rho_0\|^2_{L^4}\kappa t\left(  \int_{5\kappa^\beta}^\infty \frac{r}{(r-4\kappa^\beta)^{4}}\rmd r\right)^{1/2}\lesssim \|\rho_0\|^2_{L^4}\kappa t\left( \kappa^{-2\beta}\right)^{1/2}\notag\\
&
\lesssim \kappa^\beta \kappa t\left( \kappa^{-2\beta}\right)^{1/2} 
\ \lesssim\  \kappa^{2\beta} \kappa^{1-2\beta} t \ \lesssim \  \|\rho_0\|^2_{L^2}\kappa^\frac{q}{q+2} t. \nonumber
\end{align}
Hence, putting everything together
\begin{align}
\int_{\Omega'_\kappa}\E |\rho_0(\bxi_{t,0}(\bx))|^2\rmd \bx\lesssim  \|\rho_0\|^2_{L^2}\left[\left(\kappa^\frac{q}{q+2} t \right)^{1/2}+\kappa^\frac{q}{q+2} t\right],
\end{align}
which in turn implies
\begin{align}\label{eq:boundVar2rad}
\int_{\Omega'_\kappa} {\rm Var}\left(\rho_0(\bxi_{t,0}(\bx))\right)\rmd \bx\lesssim \int_{\Omega'_\kappa}\E |\rho_0(\bxi_{t,0}(\bx))|^2\rmd \bx
\lesssim  \|\rho_0\|^2_{L^2}\left[\left(\kappa^\frac{q}{q+2} t \right)^{1/2}+\kappa^\frac{q}{q+2} t\right].
\end{align}
Thus, \eqref{eq:intkappa} together with \eqref{eq:boundVar2rad} yields
\begin{align}
\int_{\Omega} {\rm Var}\left(\rho_0(\bxi_{t,0}(\bx)\right)\rmd \bx
\lesssim \|\rho_0\|^2_{L^2}\left[\left(\kappa^\frac{q}{q+2} t \right)^{1/2}
+\kappa^\frac{q}{q+2} t\Bigg|\ln\left(\frac{25-4\kappa^\frac{q}{q+2} t}{1-4\kappa^\frac{q}{q+2} t}\right)\Bigg|+\left|\ln\left(1-\frac{4\kappa^\frac{q}{q+2} t}{25}\right)\right|\right],
\end{align}
for all  $ t<\kappa^{-\frac{q}{q+2}}/4$, which in turn implies the upper bound
\begin{align}\label{eq:finalbourad}
\kappa \int_0^t \int_{\Omega} |\nabla \rho(s,\bx)|^2 \rmd \bx\rmd s 
\lesssim \|\rho_0\|^2_{L^2}\left[\left(\kappa^\frac{q}{q+2} t \right)^{1/2}
+\kappa^\frac{q}{q+2} t\left|\ln\left(\frac{25-4\kappa^\frac{q}{q+2} t}{1-4\kappa^\frac{q}{q+2} t}\right)\right|+\left|\ln\left(1-\frac{4\kappa^\frac{q}{q+2} t}{25}\right)\right|\right],
\end{align}
for all  $ t<\kappa^{-\frac{q}{q+2}}/4$. We are now in the position of applying Lemma \ref{lem:crucial} and conclude the proof of the third part of
Theorem \ref{thm:main} (namely, sharpness of the timescale in the radial case up to the log correction).
\vspace{2mm}

\noindent \textbf{Acknowledgements.} The research of MCZ was partially supported by the Royal Society through a 
University Research Fellowship (URF\textbackslash R1\textbackslash 191492).
The research of  TD was partially supported by the NSF DMS-1703997 grant.
We would like to thank Giuseppe Cannizzaro, Franco Flandoli, Martin Hairer and Grigorios Pavliotis for inspiring discussions.

\bibliographystyle{abbrv}
\bibliography{biblio}

\end{document}